\documentclass[11pt]{amsart}
\usepackage{amsmath,amsthm, amsfonts, amssymb, latexsym,verbatim,bbm,longtable}
\input xy
\xyoption{all}
\usepackage{hyperref,multirow}
\usepackage{tikz, ifthen}
\usepackage[shortlabels]{enumitem}
\usepackage{young}
\usepackage{mathtools}
\usepackage{dynkin-diagrams}

\allowdisplaybreaks[2] \textwidth15.1cm \textheight22cm \headheight12pt \oddsidemargin.4cm
\theoremstyle{plain}
\newtheorem{theorem}{Theorem}[section]
\newtheorem{thm}[theorem]{Theorem}
\newtheorem{lemma}[theorem]{Lemma}

\newtheorem{prop}[theorem]{Proposition}
\newtheorem{example}[theorem]{Example}

\newtheorem{defn}[theorem]{Definition}

\newtheorem{corollary}[theorem]{Corollary}
\newtheorem{cor}[theorem]{Corollary}

\newtheorem{hypothesis}[theorem]{Hypothesis}
\newtheorem{question}[theorem]{Question}

\newcommand{\iso}{\cong}
\newcommand{\arr}{\rightarrow}

\DeclareMathOperator{\Hom}{Hom}

\newcommand{\C}{\mathbb{C}}
\newcommand{\Z}{\mathbb{Z}}

\newcommand{\Q}{\mathbb{Q}}

\newcommand{\mcB}{\mathcal{B}}

\newcommand{\mcG}{\mathcal{G}}

\newcommand{\mcP}{\mathcal{P}}

\newcommand{\mbP}{\mathbb{P}}

\newcommand{\mcU}{\mathcal{U}}
\newcommand{\mcR}{\mathcal{R}}
\newcommand{\mcX}{\mathcal{X}}

\newcommand{\mfg}{\mathfrak{g}}

\newcommand{\mfh}{\mathfrak{h}}

\newcommand{\mfn}{\mathfrak{n}}

\DeclareMathOperator{\rank}{rank}

\DeclareMathOperator{\cl}{cl}

\DeclareMathOperator{\Aut}{Aut}

\DeclareMathOperator{\Gr}{Gr}

\DeclareMathOperator{\ad}{ad}

\DeclareMathOperator{\Ch}{Ch}

\DeclareMathOperator{\supp}{S}

\DeclareMathOperator{\Supp}{Supp}
\DeclareMathOperator{\Red}{RW}
\DeclareMathOperator{\Isom}{Isom}

\begin{document}

\title{The isomorphism problem for Schubert varieties}

\author{Edward Richmond}
\email{edward.richmond@okstate.edu}

\author{William Slofstra}
\email{weslofst@uwaterloo.ca}

\begin{abstract}
Schubert varieties in the full flag variety of Kac-Moody type are indexed by
elements of the corresponding Weyl group. We give a practical criterion for
when two such Schubert varieties (from potentially different flag varieties)
are isomorphic, in terms of the Cartan matrix and reduced words for the
indexing Weyl group elements. As a corollary, we show that two such Schubert
varieties are isomorphic if and only if there is an isomorphism between their
integral cohomology rings that preserves the Schubert basis.
\end{abstract}
\maketitle

\section{Introduction}


Kac-Moody flag varieties are central objects of study in geometry, topology,
and representation theory. In particular, the finite-type Kac-Moody flag
varieties are the usual generalized flag varieties associated to semisimple Lie
groups. While generalized flag varieties are finite-dimensional, Kac-Moody flag
varieties of non-finite type are infinite-dimensional. However, all Kac-Moody
flag varieties can be realized as ind-varieties stratified by
finite-dimensional Schubert varieties \cite{Kumar87}. Schubert varieties are
themselves important examples of algebraic varieties, and their singularities
are closely connected with the representation theory of the corresponding
Kac-Moody groups and algebras. As a result, their geometry has been closely
studied (see, for instance, the surveys
\cite{Billey-Lakshmibai00,Abe-Billey16}).

We are interested in the isomorphism problem for Schubert varieties: when are
two Schubert varieties isomorphic as algebraic varieties?  This natural
geometric question was first raised for Schubert varieties by Develin, Martin,
and Reiner \cite{Develin-Martin-Reiner07}. They show that two partition
varieties (a subclass of type A Schubert varieties introduced by Ding
\cite{Ding97}) are isomorphic if and only if their integral cohomology rings
are isomorphic. Aside from this, and the case of toric Schubert varieties which
we cover below, the question does not seem to have been pursued further, even
in type A or other finite types. In this paper, we give a complete solution to
the isomorphism problem for Schubert varieties in full flag varieties of any
Kac-Moody type over $\C$.

To describe this solution, it is helpful to recall some basic facts about
Schubert varieties. We follow the conventions from \cite{Kumar02}. Recall that
the starting data for constructing a Kac-Moody Lie algebra (and subsequently
it's Kac-Moody group, full flag variety, and Schubert varieties) is a
\textbf{(generalized) Cartan matrix}, which is an integer matrix
$A:=[A_{st}]_{(s,t)\in S^2}$ indexed by some finite set $S$, such that for any
$s,t\in S$,
\begin{enumerate}\label{list:cartan}
    \item $A_{st} = 2$ if $s=t$,
    \item $A_{st} \leq 0$ if $s\neq t$, and
    \item $A_{st} = 0$ if and only if $A_{ts} = 0$.
\end{enumerate}
Let $\mcG:=\mcG(A)$ and $\mcB:=\mcB(A)$ denote the Kac-Moody group and Borel
subgroup of a Cartan matrix $A$. The full flag variety corresponding to $A$ is
the quotient $\mcX=\mcX(A):=\mcG/\mcB$. The Weyl group $W := W(A)$ of $A$ is the
crystallographic Coxeter group generated by $S$, and satisfying relations
$(st)^{m_{st}} = e$, where $m_{ss} = 1$, and
\begin{equation*}
    m_{st} = \begin{cases} 2 & \text{ if } A_{st} A_{ts} = 0 \\
                           3 & \text{ if } A_{st} A_{ts} = 1 \\
                           4 & \text{ if } A_{st} A_{ts} = 2 \\
                           6 & \text{ if } A_{st} A_{ts} = 3 \\
                          \infty & \text{ if } A_{st} A_{ts} \geq 4
            \end{cases}
\end{equation*}
for all $s,t\in S$. The pair $(W,S)$ forms a Coxeter system, and the elements
of $S$ are referred to as the \textbf{simple reflections} (or \textbf{simple
transpositions}) in $W$.  For every $w \in W$, the \textbf{Schubert variety
$X(w,A)$} is defined to be the closure of $\mcB w \mcB / \mcB$ in $\mcX$. It is
well known that $X(w,A)$ is an irreducible finite dimensional complex variety
of dimension $\ell(w)$, where $\ell:W\arr \Z_{\geq 0}$ is the length function
of the Coxeter system $(W,S)$.  A product $w = s_1 \cdots s_k$ of simple
reflections $s_1,\ldots,s_k$ in $W$ is a \textbf{reduced word} if $k =
\ell(w)$. Every element of $W$ can be written as a reduced word.

If $w=s$ is a simple reflection, then $X(w,A) \iso \mbP^1$, and hence all one-dimensional
Schubert varieties are isomorphic, independent of $A$. The case of
two-dimensional Schubert varieties is more interesting:
\begin{example}\label{Ex:hirzebruch}
    Suppose $X(w,A)$ is a two-dimensional Schubert variety, so $w = st \in S$
    for some $s \neq t$. Then $X(w,A)$ is a Zariski-locally-trivial $\mbP^1$-bundle
    over $\mbP^1$. The cohomology ring $H^*(X(w,A); \Z)$ is the free $\Z$-module
    generated by the Schubert classes $\xi_u$ for $u \in \{e,s,t,w\}$, where
    $\xi_u$ has degree $2 \ell(u)$ and $e$ is the identity. The ring structure
    is determined by the relations $\xi_s^2=0$, $\xi_s
    \cdot \xi_t = \xi_w$, and $\xi_t^2 = -A_{st} \xi_w$. From this, it follows
    that $X(w,A)$ is the Hirzebruch surface $\Sigma_n$, where $n = -A_{st}$.

    It is well-known that $\Sigma_n \iso \Sigma_m$ if and only if $m=n$.  Hence
    $X(w,A) \iso X(w',A')$ with $w'=s' t'$ if and only if $A_{st} = A'_{s't'}$.
    In particular, the isomorphism type of $X(w,A)$ depends only on the value
    of $A_{st}$, not on $A_{ts}$.
\end{example}

Let $\leq$ denote the Bruhat order for the Coxeter system $(W,S)$, and
define the \textbf{support} of an element $w \in W$ to be the set
\begin{equation*}
    S(w) := \{ s\in S : s \leq w\}.
\end{equation*}
A simple reflection $s \in S$ belongs to $S(w)$ if and only if $s$ appears in
some (or equivalently, every) reduced word of $w$. Inspired by the example
of Hirzebruch surfaces, we define:
\begin{defn}\label{D:Cartan_equiv}
    Let $A$ and $A'$ be Cartan matrices over $S$ and $S'$ respectively.  Let $w \in W(A)$ and $w' \in W(A')$. We say the pair $(w,A)$ and $(w',A')$ are \textbf{Cartan equivalent}
   if there is a bijection $\sigma:S(w) \arr S(w')$ such that the following are satisfied:
    \begin{enumerate}[(a)]
        \item There are reduced words $w = s_{1} \cdots s_{k}$ and $w'=t_{1}\cdots t_{k}$ such that
            $\sigma(s_i)=t_i$ for all $1 \leq i \leq k$.
        \item If $s s' \leq w$ for $s \neq s' \in S(w)$, then $A_{ss'} = A'_{\sigma(s)\sigma(s')}$.
    \end{enumerate}
\end{defn}
Although it's not immediately obvious from the definition, we show in the next
section (see Corollary \ref{C:equivalence}) that Cartan equivalence is an equivalence
relation (and in particular is symmetric). We also refer to the bijection $\sigma: S(w)\rightarrow S(w')$ in Definition \ref{D:Cartan_equiv} as a Cartan equivalence.

Our main result is that, somewhat surprisingly, two Schubert
varieties $X(w,A)$ and $X(w',A')$ are isomorphic if and only if $(w,A)$
and $(w',A')$ are Cartan equivalent. As part of the proof, we
give a cohomological characterization of isomorphism as well.
Recall that Schubert varieties are stratified by their Schubert subvarieties,
and this stratification implies that the Schubert classes form a basis for the
integral cohomology ring $H^*(X(w,A)) := H^*(X(w,A);\Z)$.
\begin{thm}\label{T:iso}
    Let $A$ and $A'$ be Cartan matrices over $S$ and $S'$ respectively. Let $w \in W(A)$ and $w' \in W(A')$. Then the following are equivalent.
    \begin{enumerate}
    \item The pairs $(w,A)$ and $(w',A')$ are Cartan equivalent.
    \item The Schubert varieties $X(w,A)$ and $X(w',A')$ are algebraically isomorphic.
    \item There is a graded ring isomorphism $\phi:H^*(X(w,A))\rightarrow H^*(X(w',A'))$ which
        sends the Schubert basis for $H^*(X(w,A))$ to the Schubert basis for $H^*(X(w',A'))$.
    \end{enumerate}
\end{thm}
Theorem \ref{T:iso} can be readily applied in many situations. To illustrate
this, we give several examples and applications in Section
\ref{SS:examples_and_apps}. We remark that the theorem does not hold if we
drop the requirement in part (3) that the isomorphism preserve Schubert bases.
Indeed, returning to Example \ref{Ex:hirzebruch}, two Hirzebruch surfaces
$\Sigma_n$ and $\Sigma_m$ have isomorphic integral cohomology rings if and only
if they are diffeomorphic, which happens if and only if $m = n \mod 2$.

Hirzebruch surfaces are also examples of toric varieties. In general, a
Schubert variety $X(w,A)$ is a toric variety if and only if $w = s_1 \cdots
s_k$ for distinct simple reflections $s_1,\ldots,s_k \in S$
\cite{Karuppuchamy13}. It follows from this that toric Schubert varieties are
toric manifolds (smooth compact toric varieties).  It is well known that
isomorphism classes of toric varieties are determined by the combinatorial data
in their associated fans \cite{Fulton93}.  In addition, a result of Masuda
states that toric manifolds are isomorphic if and only if their equivariant
cohomology rings are weakly isomorphic \cite{Masuda08}.  Both of these criteria apply to toric
Schubert varieties in particular. It is an open question as to whether toric
manifolds are cohomologically rigid, in the sense that any two toric manifolds
with isomorphic cohomology rings are diffeomorphic or homeomorphic. We can
ask the same question for Schubert varieties:
\begin{question}\label{Q:rigid}
    Suppose $X(w,A)$ and $X(w',A')$ are both smooth, and  $H^*(X(w,A))$ and
    $H^*(X(w',A'))$ are isomorphic as graded rings. Are $X(w,A)$ and
    $X(w',A')$ diffeomorphic or homeomorphic?
\end{question}
If integral cohomology is replaced with rational cohomology $H^*(X(w,A); \Q)$,
then Question \ref{Q:rigid} has a negative answer, as all Hirzebruch surfaces have
isomorphic cohomology rings over $\Q$. Another counterexample is provided by
the flag varieties of finite types $B_n$ and $C_n$, since the cohomology rings
are isomorphic over $\Q$, but not over $\Z$
\cite{Bergeron-Sottile02,Borel53,Edidin-Graham}.

Finally, recall that if $A$ is a Cartan matrix over $S$, then for any
subset $J \subseteq S$ there is a partial flag variety $\mcX^J := \mcG(A) / \mcP(A)_J$,
where $\mcP(A)_J$ is the parabolic subgroup generated by the elements of $J$.
Partial flag varieties are also stratified by Schubert varieties $X^J(w,A)$,
where $X^J(w,A)$ is defined as the closure of $\mcB w \mcP_J / \mcP_J$ in $\mcX^J$.
Partial flag varieties include familiar examples such as the Grassmannians.
Next we show the notion of Cartan equivalence is neither necessary nor sufficient for distinguishing Schubert
varieties in partial flag varieties.

\begin{example}\label{Ex:partial_counter}
    Consider the Cartan matrices of types $A_3$ and $C_2$ given by
    \begin{equation*}
        A_3 = \left(\begin{matrix}
            2 & -1 & 0\\
            -1 & 2 & -1\\
            0 & -1 & 2
        \end{matrix}\right)\quad\text{and}\quad
        C_2 = \left(\begin{matrix} 2 & -2 \\
        -1 & 2  \end{matrix}\right)
    \end{equation*}
    over index sets $S=\{s_1,s_2,s_3\}$ and $S'=\{s'_1,s'_2\}$ respectively.
    First consider the case where $J = \{s_2,s_3\}$ and $J' = \{s'_2\}$ with $w=s_3s_2 s_1$ and $w'=s'_1 s'_2 s'_1$.
    Then $X^J(w, A_3)$ and $X^{J'}(w', C_2)$ are both isomorphic to the projective space $\mbP^3$.  However $|S(w)|=3$ and $|S(w')|=2$ and hence $(w,A_3)$ and $(w',C_3)$ cannot be Cartan equivalent.

    \smallskip

    Conversely, consider $w = s_2 s_1 s_3 s_2$ and $J=\{s_1,s_3\}$.  Then $X^{\emptyset}(w,A_3)
    = X(w,A_3)$ is singular, whereas $X^{J}(w,A_3)$ is the Grassmannian $\Gr(2,4)$,
    and hence is smooth.  Clearly $(w,A_3)$ is Cartan equivalent to itself; however $X^{\emptyset}(w,A_3)\ncong X^{J}(w,A_3)$.
    \end{example}

The class of Schubert varieties of partial flag varieties is much broader than
the class of Schubert varieties in full flag varieties.  It is not clear
whether the isomorphism problem for this broader class should have a simple
combinatorial solution like Cartan equivalence. We leave
this as an open problem.




In proving Theorem \ref{T:iso}, we need to show that a Cartan equivalence can
be constructed from the isomorphism between cohomology rings.  To do this, we
prove that the Cartan matrix entries $A_{st}$ for $st \leq w$ and the reduced
words for $w$ can be recovered from $H^*(X(w),A)$ along with its Schubert
basis.  This gives a procedure to solve a related problem of independent
interest: constructing a presentation of a Schubert variety (as a Schubert
variety) solely from geometric data.  We outline this procedure in Section \ref{SS:cohomology2}.

\subsection{Examples and applications of Theorem \ref{T:iso}}\label{SS:examples_and_apps}
In this section, we illustrate the potential applications of Theorem \ref{T:iso}
with several examples. We start with a basic example of how the theorem works:

\begin{example}\label{Ex:Isom_0}
Let $A$ be the following Cartan matrix over $S=\{s_1,s_2,s_3,s_4\}$:
$$\left(\begin{matrix}
2 & -1 & -1 & 0\\
-1 & 2 & -3 & 0\\
-2 & -1 & 2 & -5\\
0 & 0 & -3 & 2
\end{matrix}\right)$$
Let $w$ be the reduced word $s_2s_3s_2s_1s_4$ in $W(A)$.  To illustrate Definition \ref{D:Cartan_equiv} part (2), we list all elements $s_is_j\leq w$ in the corresponding positions within the matrix $A$ and then highlight the relevant data in $A$ needed to determine the Cartan equivalence class of $(w,A)$.
\begin{equation}\label{Eq:Cartan_example}
\left(\begin{matrix}
 &  &  & s_1s_4\\
s_2s_1 &  & s_2s_3 & s_2s_4\\
s_3s_1 & s_3s_2 &  & s_3s_4\\
s_4s_1 & s_4s_2 &  &
\end{matrix}\right)\quad \Rightarrow \quad
\left(\begin{matrix}
2 & * & * & 0\\
-1 & 2 & -3 & 0\\
-2 & -1 & 2 & -5\\
0 & 0 & * & 2
\end{matrix}\right)
\end{equation}
The ``$*$'' entries correspond to pairs of indices $(s_i, s_j)$, $i \neq j$,
such that $s_i s_j \not\leq w$.
Suppose $A'$ is another Cartan matrix over $S$ which agrees with
$A$ on all the non-starred entries. By Lemma \ref{L:cartan}, the word
$s_2 s_3 s_2 s_1 s_4$ will be reduced in $W(A')$ as well, so $(w,A)$
and $(w,A')$ will be Cartan equivalent. By Theorem \ref{T:iso}, we would then
have $X(w,A)\iso X(w',A')$.

More generally, if a Cartan matrix $A'$ contains a submatrix of the form on the
right in Equation \eqref{Eq:Cartan_example} (where the ``$*$'' entries can be
any number), then there will be an element $w'\in W(A')$ such that $(w',A')$ is
Cartan equivalent to $(w,A)$, and hence $X(w,A) \iso X(w',A')$.
\end{example}

One of the advantages of Theorem \ref{T:iso} is that it makes it easy to determine
whether Schubert varieties in different types are isomorphic. For instance:
\begin{example}\label{Ex:Isom_1}
Consider the Cartan matrices
$$A_3=\left(\begin{matrix}
2 & -1 & 0\\
-1 & 2 & -1\\
0 & -1 & 2
\end{matrix}\right)\quad\text{and}\quad B_3=\left(\begin{matrix}
2 & -1 & 0\\
-1 & 2 & -1\\
0 & -2 & 2
\end{matrix}\right)$$
of types $A_3$ and $B_3$ over $S=\{s_1, s_2, s_3\}$.  Theorem \ref{T:iso} implies that $$X(s_1s_2s_3, A_3)\cong X(s_1s_2s_3, B_3)$$ and $$X(s_3s_2s_1, A_3)\ncong X(s_3s_2s_1, B_3).$$ Also note that
$$X(s_2s_1s_3, A_3)\ncong X(s_1s_3s_2, A_3).$$
Note that every Schubert variety in this example is smooth and has the same Poincar\'{e} polynomial.  Hence the varieties in this example cannot be distinguished by these properties. In the last example, we have a case where $X(w,A)\ncong X(w^{-1},A)$.
\end{example}

For any $A$ is a Cartan matrix indexed by $S$, and $J \subseteq S$, let
$A_J:=[A_{st}]_{(s,t)\in J^2}$ denote the induced Cartan matrix over $J$. The
group $W(A_J)$ can be thought of as the subgroup of $W(A)$ generated by $J$,
and is typically denoted by $W_J$.  If $w\in W(A)$, then $w\in W(A_{S(w)})$.
Something that shows up in the previous example is that the isomorphism type of
$X(w,A)$ depends only on $A_{S(w)}$.  In fact, $X(w,A) \iso X(w,A_{S(w)})$.
While this follows from Theorem \ref{T:iso}, it can also be easily proved
without it (see for instance \cite[Lemma 4.8]{Richmond-Slofstra16}). We say $w\in W(A)$ is \textbf{fully supported} if $S(w) = S$.

Finite and affine type Cartan matrices are classified by Dynkin diagrams, which
are graphs with simple edges and decorated multiedges (See Figure \ref{Figure:Dynkin_diagrams}). The vertex set of the
Dynkin diagram is the index set $S$ of the Cartan matrix $A$, and the edge or
multiedge between vertices $s$ and $t$ determines the matrix values $A_{st}$ and $A_{ts}$. For
any Cartan matrix $A$, we can also consider the Coxeter graph of $A$,
which is the graph with vertex set $S$, and $m_{st}-2$ edges between vertices
$s$ and $t$.  In finite and affine types, the Coxeter graph is the Dynkin
diagram with decorations removed from the multiedges.

\begin{figure}[p]
$$
\begin{tabular}{|c|c|c|c|}\hline
Finite type &Dynkin diagram & $\Aut(A)$ & $\Aut(\Gamma(A))$ \\ \hline
$A_n$ & $\dynkin A{}$& $\Z_2$& $\Z_2$ \\ \hline
$B_n\, (n\geq 3)$ & $\dynkin B{}$& $\Z_1$ & $\Z_2$ \\ \hline
$C_n\, (n\geq 2)$ & $\dynkin C{}$& $\Z_1$ & $\Z_2$ \\ \hline
$D_4$ & $\dynkin D4{}$ & $S_3$ & $S_3$ \\ \hline
$D_n\, (n\geq 5)$ & $\dynkin D{}$& $\Z_2$& $\Z_2$\\ \hline
$E_6$ & $\dynkin E6$& $\Z_2$ &$\Z_2$\\ \hline
$E_7$ & $\dynkin E7$&$\Z_1$ &$\Z_1$\\ \hline
$E_8$ & $\dynkin E8$&$\Z_1$ &$\Z_1$\\ \hline
$F_4$ & $\dynkin F4$&$\Z_1$ &$\Z_2$\\ \hline
$G_2$ & $\dynkin G2$&$\Z_1$ &$\Z_2$\\ \hline
\end{tabular}
$$

$$
\begin{tabular}{|c|c|c|c|}\hline
Affine type &Dynkin diagram & $\Aut(A)$ & $\Aut(\Gamma(A))$ \\ \hline
$\tilde A_n$ & $\dynkin[extended]A{}$& $I_2(n)$ & $I_2(n)$\\ \hline
$\tilde B_n\, (n\geq 3)$ & $\dynkin[extended]B{}$& $\Z_1$& $\Z_2$\\ \hline
$\tilde C_n\, (n\geq 2)$ & $\dynkin[extended]C{}$& $\Z_2$&$\Z_2$\\ \hline
$\tilde D_4$ & $\dynkin[extended]D4{}$ & $S_4$ & $S_4$\\ \hline
$\tilde D_n\, (n\geq 5)$ & $\dynkin[extended]D{}$ & $\Z_2\times\Z_2\times\Z_2$& $\Z_2\times\Z_2\times\Z_2$\\ \hline
$\tilde E_6$ & $\dynkin[extended]E6{}$ & $S_3$ & $S_3$\\ \hline
$\tilde E_7$ & $\dynkin[extended]E7{}$&$\Z_2$ &$\Z_2$\\ \hline
$\tilde E_8$ & $\dynkin[extended]E8{}$&$\Z_1$&$\Z_1$\\ \hline
$\tilde F_4$ & $\dynkin[extended]F4{}$&$\Z_1$&$\Z_2$\\ \hline
$\tilde G_2$ & $\dynkin[extended]G2{}$&$\Z_1$&$\Z_2$\\ \hline
$\tilde A^2_{2}$ & \dynkin[extended]A[2]2 & $\Z_1$ & $\Z_2$\\ \hline
$\tilde A^2_{2n}$ & \dynkin[extended]A[2]{even} & $\Z_1$ & $\Z_2$\\ \hline
$\tilde A^2_{2n+1}$ & \dynkin[extended]A[2]{odd} & $\Z_2$ & $\Z_2$\\ \hline
$\tilde D^2_{n}$ & \dynkin[extended]D[2]{} & $\Z_2$ & $\Z_2$\\ \hline
$\tilde D^3_{4}$ & \dynkin[extended]D[3]4 & $\Z_1$ & $\Z_2$\\ \hline
$\tilde E^2_{6}$ & \dynkin[extended]E[2]{6} & $\Z_1$ & $\Z_2$\\ \hline
\end{tabular}
$$
\caption{
Dynkin diagrams and automorphism groups of all finite and affine Lie types.
$S_n$ denotes the symmetric group, and $I_2(n)$ denotes the dihedral group.}\label{Figure:Dynkin_diagrams}
\end{figure}

A particularly interesting case to look at is finite versus infinite type.
Recall that a Cartan matrix $A$ is \textbf{simply-laced} if
$A_{st} \in \{0,-1\}$ for all $s \neq t \in S$. Equivalently, a Cartan matrix
is simply-laced if the exponents $m_{st}$ of the Coxeter relations are either
$2$ or $3$, i.e. the Coxeter graph is simple. Given a Cartan matrix $A$, it is
convenient to define the \textbf{simple Coxeter graph} $\Gamma(A):=(S,E)$
to be the graph with vertex set $S$ and edges $(s,t)\in E$ if and only
if $A_{st} \neq 0$. In other words, $\Gamma(A)$ is the underlying simple graph
of the Coxeter graph of $A$.
\begin{lemma}\label{L:simple}
    If $X(w,A) \iso X(w',A')$, then $\Gamma(A_{S(w)}) \iso \Gamma(A_{S(w')})$.
\end{lemma}
\begin{proof}
    Let $\sigma : S(w) \arr S(w')$ be a Cartan equivalence. If $s,t \in S(w)$,
    then either $st \leq w$ or $ts \leq w$. If $st \leq w$, then $A_{st} =
    A_{\sigma(s)\sigma(t)}'$, while if $ts \leq w$ then $A_{ts} =
    A_{\sigma(t)\sigma(s)}'$. Using the fact that $A_{st} = 0$ if and only
    if $A_{ts} = 0$, we conclude that $A_{st}=0$ if and only if $A_{\sigma(s)\sigma(t)}'=0$.  Therefore $\sigma$ is a graph isomorphism.
\end{proof}

\begin{cor}\label{C:simplylaced}
    If $X(w,A)$ is isomorphic to a Schubert variety in a finite type flag variety,
    then $\Gamma(A_{S(w)})$ is a finite type Coxeter graph. In particular, if $A_{S(w)}$
    is simply-laced, then $X(w,A)$ is isomorphic to a Schubert variety of finite
    type if and only if $A_{S(w)}$ is of finite type.
\end{cor}
\begin{proof}
    It is easy to verify from Figure \ref{Figure:Dynkin_diagrams} that the
    underlying simple graph of a finite type Dynkin diagram is also a finite
    type Dynkin diagram. If $X(w,A) \iso X(w',A')$ where $A'$ has finite type,
    then by Lemma \ref{L:simple}, the graph $\Gamma(A_{S(w)})$ is isomorphic to $\Gamma(A'_{S(w')})$
    and hence has finite type. If $A_{S(w)}$ is simply-laced, then $\Gamma(A_{S(w)})$
    is the Coxeter graph of $A_{S(w)}$, and hence $A_{S(w)}$ is of finite type.
\end{proof}
\begin{example}
    The Coxeter graph of affine type $\tilde{A}_n$ is a cycle for $n \geq 2$.  Hence if $w \in W(\tilde{A}_n)$ is fully supported, then $X(w,\tilde{A}_n)$ is
    not isomorphic to any Schubert variety of finite type. The same is true
    for affine types $\tilde{D}_n$ and $\tilde{E}_n$.
\end{example}

\begin{example}
    If $A_{S(w)}$ is not simply-laced, then it's possible for $X(w,A_{S(w)})$
    to be isomorphic to a finite-type Schubert variety, even if $A_{S(w)}$
    is not of finite type. For instance, consider the Cartan matrices
    \begin{equation*}
        \tilde{A}_1 = \begin{pmatrix} 2 & -2 \\ -2 & 2 \end{pmatrix}
        \text{ and } C_2 = \begin{pmatrix} 2 & -2 \\ -1 & 2 \end{pmatrix}
    \end{equation*}
    of affine type $\tilde{A}_1$ and finite type $C_2$ over index set $\{s_1,s_2\}$.
    Then $X(s_1 s_2,\tilde{A}_1) \iso X(s_1 s_2, C_2)$.
\end{example}

\begin{example}
    The criterion in Definition \ref{D:Cartan_equiv} simplifies if $A$ is
    simply-laced, since $A_{st} = A_{ts}$ for all $s,t$. More generally, suppose
    that $A$ and $A'$ are symmetric Cartan matrices, in the sense that $A_{st}
    = A_{ts}$ for all $s, t \in S$,\footnote{Symmetric is not the same as
symmetrizable, a common condition imposed on Cartan matrices when studying the
representation theory of Kac-Moody Lie algebras.}
and let $w \in W(A)$, $w' \in W(A')$.  If $X(w,A) \iso X(w',A')$, then
    the Cartan matrices  $A_{S(w)}$ and $A_{S(w')}'$ are the same up to
    permutation of rows and columns. In other words, if $X(w,A)$
    and $X(w',A')$ are
    isomorphic, then the flag varieties $\mcX(A_{S(w)})$ and $\mcX(A_{S(w')})$
    are isomorphic, with an isomorphism that identifies $X(w,A)$ and $X(w',A')$.
\end{example}

We can use Lemma \ref{L:simple} for comparison to other classes of Schubert varieties
as well.
\begin{example}
    The finite type $A$ Schubert varieties are the best studied class of
    Schubert varieties.  By Lemma \ref{L:simple}, a Schubert variety $X(w,A)$
    is isomorphic to a Schubert variety of finite type $A$ if and only if
    \begin{enumerate}
    \item the simple Coxeter graph $\Gamma(A_{S(w)})$ is a disjoint union of paths, and
    \item for every $st\leq w$, we have $A_{st}\in\{0,-1\}$.
    \end{enumerate}
\end{example}
We can also use Theorem \ref{T:iso} to calculate the isomorphism classes of Schubert
varieties in a fixed Kac-Moody flag variety. For any Cartan matrix $A$ and $w
\in W(A)$, let $$\Isom(w,A):=\{w'\in W(A) : X(w,A)\iso X(w',A)\}$$ denote the
isomorphism class of $X(w,A)$ within the Kac-Moody flag variety $\mcX(A).$
\begin{example}\label{Ex:Isom_2}
    The flag variety $\mcX(A_3)$ (where $A_3$ is the Cartan matrix over
    $S=\{s_1,s_2,s_3\}$ as in Example \ref{Ex:Isom_1}) has $14$ isomorphism classes
    of Schubert varieties:
$$\begin{tabular}{|c|c|}
\hline
$\ell(w)$ & $\Isom(w,A_3)$\\ \hline
$0$ & $\{1\}$\\ \hline
$1$ & $ \{s_1,\, s_2,\, s_3\}$\\ \hline
$2$ & $\{s_1s_3\},\quad \{s_1s_2,\, s_2s_1,\, s_2s_3,\, s_3s_2\}$\\ \hline
$3$ & $\{s_1s_2s_1,\, s_2s_3s_2\},\quad \{s_1s_3s_2\},\ \{s_2s_1s_3\},\quad \{s_1s_2s_3,\, s_3s_2s_1\} $ \\ \hline
$4$ & $\{s_1s_2s_3s_2,\, s_3s_2s_1s_2\},\quad \{s_2s_1s_2s_3,\, s_2s_3s_2s_1\},\quad \{s_2s_1s_3s_2\}$\\ \hline
$5$ & $\{s_2s_1s_2s_3s_2,\, s_2s_3s_2s_1s_2\},\quad \{s_3s_2s_1s_2s_3\}$\\ \hline
$6$ & $\{s_3s_2s_1s_3s_2s_3\}$\\ \hline
\end{tabular}$$
In types $A_4$ and $A_5$ there are $54$ and $316$ Schubert isomorphism classes respectively.
\end{example}
When $w \in W(A)$ is not fully supported, it's possible for $X(w,A)$ to be
isomorphic to $X(w',A)$ where $S(w) \neq S(w')$. For instance, $X(s, A) \iso
X(t,A)$ for all $s,t \in S$.  So in Example \ref{Ex:Isom_2}, $|\Isom(s_1,A_3)|
= 3$, and in general $|\Isom(s,A)| = |S|$.  However, in Example \ref{Ex:Isom_2}
all the fully supported elements $w$ have isomorphism classes of size $1$ or
$2$. In fact, the sets $\Isom(w,A)$ when $A$ is of finite and affine type and
$w$ is fully supported are surprisingly small. To explain this, we can bound
the size of the sets $\Isom(w,A)$ in terms of the automorphism groups of $\Gamma(A)$.
Recall that a \textbf{diagram automorphism} of a Cartan matrix $A$ with index set $S$
is a bijection $\sigma:S\arr S$ such that $A_{st} = A_{\sigma(s) \sigma(t)}$.
We let $\Aut(A)$ denote the group of diagram automorphisms of $A$.  Diagram
automorphisms play an important role in the classification of automorphisms of
Kac-Moody groups \cite{Caprace-Muhlherr05,Carter-Chen93,Kac-Wang92}. We let
$\Aut(\Gamma(A))$ denote the graph automorphism group of $\Gamma(A)$, i.e. the
set of bijections $\sigma : S \arr S$ such that $A_{st} = 0$ if and only if
$A_{\sigma(s)\sigma(t)} = 0$. Since $\Gamma(A)$ is the Dynkin diagram of the
Cartan matrix where every off-diagonal non-zero entry of $A$ is changed to $-1$,
$\Aut(\Gamma(A))$ is also the automorphism group of some Cartan matrix.

\begin{corollary}\label{C:Isom_class_Dynkin_autos}
    If $A$ is a Cartan matrix over $S$ and $w\in W(A)$ is fully supported, then
    \begin{equation*}
        |\Isom(w,A)| \leq |\Aut(\Gamma(A))|.
    \end{equation*}
    Furthermore, if $A$ is symmetric, then
    \begin{equation*}
        |\Isom(w,A)| \leq |\Aut(A)|.
    \end{equation*}
\end{corollary}
\begin{proof}
    If $w \in W(A)$ is fully-supported, then a Cartan equivalence $\sigma : S \arr S$
    between $(w,A)$ and $(w',A)$ is an element of $\Aut(\Gamma(A))$.

    If $A$ is symmetric,
    and $s \neq t \in S$, then either $st \leq w$, in which case $A_{st} = A_{\sigma(s)\sigma(t)}$,
    or $ts \leq w$, in which case $A_{st} = A_{ts} = A_{\sigma(t)\sigma(s)} = A_{\sigma(s)\sigma(t)}$. So $\sigma$ will be in $\Aut(A)$.
\end{proof}
The automorphisms groups of $A$ for $A$ finite and (untwisted) affine types are
well-known, and are shown in Figure \ref{Figure:Dynkin_diagrams}.
As can be verified from the table, if $A$ is finite or affine,
then $\Gamma(A)$ is also finite or affine, so $\Aut(\Gamma(A))$ can be
determined from the table as well. It follows that for all simple finite and
affine types except affine type $\tilde{A}_n$, $|\Isom(w,A)| \leq 24$ for all
fully supported $w$, and for many simple finite and affine types, the bound is
$|\Isom(w,A)| \leq 2$.

\subsection{Outline of paper}
The remainder of this paper focuses on proving Theorem \ref{T:iso}.  In Section
\ref{S:combinatorics}, we review some basic facts on Coxeter groups and apply
them to Cartan equivalences.  In Section \ref{S:KacMoody_algebras} we recall
the definition of the algebraic structure on Schubert varieties using the Kac-Moody
Lie algebra, and prove (1) implies (2). In Section \ref{S:cohomology}, we study
the cohomology ring of Schubert varieties, and prove (2) implies (3) and (3)
implies (1). An explanation of how to construct a presentation of a Schubert
variety from geometric data is given in Subsection \ref{SS:cohomology2}.
 We use \cite{Kumar02} as the primary reference for background material
throughout the paper.

\subsection{Acknowledgements}
The authors thank Dave Anderson, Shrawan Kumar, Kevin Purbhoo, and Alex Yong
for helpful conversations.  WS was supported by NSERC DG 2018-03968

\section{Combinatorics of Cartan equivalence}\label{S:combinatorics}
In this section, we establish some basic combinatorial properties of Cartan
equivalence. As usual, a word over alphabet $S$ is a sequence
$(s_1,\ldots,s_k)$ with $s_i \in S$ for $1 \leq i \leq k$. As mentioned in the
introduction, if $A$ is a Cartan matrix over index set $S$, then
$(s_1,\ldots,s_k)$ is a reduced word if there is no way to write $s_1 \cdots
s_k \in W(A)$ as a product of fewer than $k$ elements of $S$.  For any $w\in
W(A)$, let $\Red(w)$ denote the set of reduced words of $w$. Normally when
working with $W(A)$, a word $(s_1,\ldots,s_k)$ is written as $s_1 \cdots s_k$,
so the expression $s_1 \cdots s_k$ can refer to a word or to an element of $W(A)$ depending
on the context. We use the same convention in this paper, but we'll also use
the sequence notation for words when there is potential for confusion.

We need the following special case of the subword property for Bruhat order.
\begin{lemma}\label{L:ij}
    Let $A$ be a Cartan matrix over a finite set $S$.  Let $w\in W(A)$ and suppose $s,t \in S(w)$. If $A_{st} < 0$ then the following are equivalent:
    \begin{itemize}
        \item $st \leq w$.
        \item The element $s$ appears before $t$ in some reduced word for $w$.
        \item The element $s$ appears before $t$ in any reduced word for $w$.
    \end{itemize}
Otherwise, if $A_{st}=0$, then $st=ts\leq w.$
\end{lemma}
\begin{proof}
    Special case of \cite[Theorem 2.2.2]{Bjorner-Brenti05}.
\end{proof}

The definition of Cartan equivalence ostensibly requires us to find a
reduced expression for $w$ which corresponds to a reduced expression for $w'$.
However, any reduced expression will do:
\begin{lemma}\label{L:cartan}
    Let $A$ and $A'$ be Cartan matrices on $S$ and $S'$ respectively and let $w
    \in W(A)$. Let $(s_{1},\ldots, s_{k})$ be a reduced word for $w \in W(A)$,
    and suppose $\sigma:S(w) \arr S'$ is an injection satisfying the condition
    that $A_{st} = A'_{\sigma(s)\sigma(t)}$ whenever $st \leq w$. Then
    $(\sigma(s_1),\ldots,\sigma(s_k))$ is a reduced word for $w' := \sigma(s_1)
    \cdots \sigma(s_k) \in W(A')$, and $(w,A)$ is Cartan equivalent to $(w',A')$.
    Furthermore, $(t_1,\ldots,t_k) \in \Red(w)$
    if and only if $(\sigma(t_1),\ldots, \sigma(t_k)) \in \Red(w')$.
\end{lemma}
\begin{proof}
    Let $m_{st}$ and $m_{st}'$ be the Coxeter exponents for $A$ and $A'$
    respectively.  A word $(t_1,\ldots,t_k)$ with $t_1,\ldots,t_k \in S'$ is
    non-reduced if and only if it is possible to apply Coxeter relations
    \begin{equation*}
        \underbrace{(s,t, \ldots)}_{\text{$m_{st}'$ times}} =
        \underbrace{(t,s, \ldots)}_{\text{$m_{st}'$ times}}
    \end{equation*}
    for simple reflections $s,t \in S'$, to get a word $(t_1',\ldots,t_k')$
    with $t_{i}' = t_{i+1}'$ for some $1 \leq i < k$. Suppose
    $(\sigma(s_1),\ldots,\sigma(s_k))$ contains an alternating subword
    $(\sigma(s),\sigma(t),\ldots)$ of length $m_{\sigma(s)\sigma(t)}'$. If $m_{\sigma(s)\sigma(t)}' \geq 3$,
    then $st \leq w$ and $ts \leq w$, and so $A_{st} = A_{\sigma(s)\sigma(t)}'$
    and $A_{ts} = A_{\sigma(t)\sigma(s)}'$. If $m_{\sigma(s)\sigma(t)}' = 2$, then
    $st \leq w$, and hence $A_{st} = A_{\sigma(s)\sigma(t)}' = 0$, and
    so $A_{ts} = 0 = A_{\sigma(t)\sigma(s)}'$ as well. Thus in either case,
    $m_{st} = m_{\sigma(s)\sigma(t)}'$, and so any Coxeter relation that
    can be applied to $(\sigma(s_1),\ldots,\sigma(s_k))$ can also be
    applied to $(s_1,\ldots,s_k)$, giving another reduced word for $w$.
    Since $(s_1,\ldots,s_k)$ is reduced, we'll never get a word of the
    form $(s_1',\ldots,s_k')$ with $s_{i}' = s_{i+1}'$ by applying Coxeter
    relations. So the same is true of $(\sigma(s_1),\ldots,\sigma(s_k))$, and
    thus $(\sigma(s_1),\ldots,\sigma(s_k))$ is reduced as well.
    The fact that $(w,A)$ and $(w',A')$ are Cartan equivalent follows immediately.

    Similarly, if $(s_1,\ldots,s_k)$ contains an alternating
    subword $(s,t,\ldots)$ of length $m_{st}$, then the same argument shows that
    $m_{st} = m_{\sigma(s)\sigma(t)}'$, and hence any Coxeter relation that can be applied
    to $(s_1,\ldots,s_k)$ can also be applied to
    $(\sigma(s_1),\ldots,\sigma(s_k))$. Since $\Red(w)$ and $\Red(w')$ are
    exactly the words we get by applying Coxeter relations to
    $(s_1,\ldots,s_k)$ and $(\sigma(s_1),\ldots,\sigma(s_k))$, it follows that
    $(t_1,\ldots,t_k) \in \Red(w)$ if and only if $(\sigma(t_1),\ldots,\sigma(t_k))\in \Red(w')$.
\end{proof}

\begin{cor}\label{C:equivalence}
    Cartan equivalence is an equivalence relation.
\end{cor}
\begin{proof}
    Clearly Cartan equivalence is reflexive. Suppose that $\sigma : S(w) \arr S(w')$
    is a Cartan equivalence as in Definition \ref{D:Cartan_equiv}. By definition,
    there is a reduced word $s_1 \cdots s_n$ for $w$ such that $\sigma(s_1) \cdots
    \sigma(s_n)$ is a reduced word for $w'$. Suppose $s't' \leq w'$, where $s'
    = \sigma(s)$ and $t' = \sigma(t)$. By Lemma \ref{L:ij}, if $A_{s't'}' \neq
    0$, then $s'$ occurs before $t'$ in $\sigma(s_1) \cdots \sigma(s_n)$. So
    $st \leq w$, and hence $A'_{s't'} = A_{st} = A_{\sigma^{-1}(s')\sigma^{-1}(t')}$.
    If $A_{s't'}' =0$, then at least one of $st$ or $ts \leq w$. In both
    cases $A_{\sigma^{-1}(s') \sigma^{-1}(t')} = A_{st} = 0 = A_{s't'}'$, since
    if $ts \leq w$, then $A_{ts} = A'_{t's'} = 0$, implying $A_{st} = 0$. So
    $\sigma^{-1}$ is also a Cartan equivalence, and Cartan equivalence is symmetric.

    For transitivity, suppose $\sigma : S(w) \arr S(w')$ and $\tau : S(w') \arr S(w'')$
    are Cartan equivalences from $(w,A)$ to $(w',A')$ and $(w',A')$ to $(w'',A'')$
    respectively. Let $s_1,\ldots,s_n$ be a reduced word for $w$. By Lemma \ref{L:cartan},
    $\sigma(s_1)\cdots\sigma(s_n)$ is a reduced word for $w'$, and therefore
    $\tau(\sigma(s_1))\cdots \tau(\sigma(s_n))$ is a reduced word for $w''$.
    If $st \leq w$, then $\sigma(s)\sigma(t) \leq w'$. So $A_{st} =
    A_{\sigma(s)\sigma(t)}' = A_{\tau(\sigma(s)),\tau(\sigma(t))}''$ for all
    $st \leq w$, and hence $\tau \circ \sigma : S(w) \arr S(w'')$ is a Cartan
    equivalence between $(w,A)$ and $(w'',A'')$.
\end{proof}

Another corollary of Lemma \ref{L:cartan} is that Cartan equivalences preserve
Bruhat intervals:
\begin{cor}\label{C:cartan}
    Suppose $(w,A)$ and $(w',A')$ are Cartan equivalent under the bijection
    $\sigma:S(w)\arr S(w')$. For any $v \leq w$, there is $v' \leq w'$ such that
    $(v,A)$ and $(v',A')$ are Cartan equivalent under an induced bijection
    $\sigma|_{S(v)} : S(v) \arr \sigma(S(v))$.   Furthermore, this correspondence
    gives a poset isomorphism $\widetilde{\sigma}: [e,w] \arr [e,w']$.
\end{cor}
\begin{proof}
    Fix a reduced word $w=s_1\cdots s_k$ and let $v\leq w$.  Then there is a
    subsequence $(i_1,\ldots,i_m)$ for which $v=s_{i_1}\cdots s_{i_m}$ is a reduced
    word for $v$.  Let $v':=\sigma(s_{i_1})\cdots\sigma(s_{i_m})$.  It follows from
    Lemma \ref{L:cartan} that $(v,A)$ is Cartan equivalent to $(v',A')$ and $v'\leq
    w'$. Also by Lemma 2.2, the element $v'$ is independent of the choice of reduced
    word for $v$. So $\sigma$ induces a function $\widetilde{\sigma} : [e,w] \arr
    [e,w']$. If $u \leq v$, then there is another subsequence $(j_1,\ldots,j_k)$
    such that $u = s_{i_{j_1}} \cdots s_{i_{j_k}}$ is a reduced word for $u$, and
    hence $\widetilde{\sigma}(u) = \sigma(s_{i_{j_1}}) \cdots \sigma(s_{i_{j_k}})
    \leq \widetilde{\sigma}(v)$, so $\widetilde{\sigma}$ is order-preserving.
    The function $\widetilde{\sigma^{-1}} : [e,w'] \arr [e,w]$ induced by
    $\sigma^{-1}$ is an inverse to $\widetilde{\sigma}$, so
    $\widetilde{\sigma}$ is a bijection.
\end{proof}

\section{Cartan equivalence and Kac-Moody Lie algebras}\label{S:KacMoody_algebras}
In this section we prove that condition (1) implies condition (2) from Theorem
\ref{T:iso}, which we state as its own proposition:
\begin{prop}\label{P:cartantoalg}
    If $(w,A)$ and $(w',A')$ are Cartan equivalent, then $X(w,A) \iso X(w',A')$ as algebraic varieties.
\end{prop}
The general idea of the proof of Proposition \ref{P:cartantoalg} is to use the Cartan equivalence between $(w,A)$ and $(w',A')$ to construct a linear map $\pi:V\arr V'$ between certain vector spaces $V$ and $V'$ for which the algebraic structures of $X(w,A)$ and $X(w',A')$ are realized as closed embeddings into $\mbP(V)$ and $\mbP(V')$ respectively.  Restricting $\pi$ to $X(w,A)$ will then give an algebraic bijection between $X(w,A)$ and $X(w',A')$.  Since these varieties are normal, they will be isomorphic.  We begin by recalling the construction of a Kac-Moody Lie algebra $\mfg(A)$ from a Cartan matrix $A$ over a finite set $S$ of size $n$ as given in \cite{Kumar02} or \cite{Kac90}.
Let $\mfh(A)$ be a complex vector space of dimension $2n-\rank(A)$.  Let $\{h_s\}_{s\in S}\subset \mfh(A)$ denote a set of linearly independent vectors with corresponding simple root vectors $\{\alpha_s\}_{s\in S}\subset \mfh^*(A)$ satisfying
$$\alpha_t(h_s)=A_{st}.$$
The \textbf{Kac-Moody Lie algebra} $\mfg(A)$ is the Lie algebra
generated by $\mfh$, along with elements $\{e_s\}_{s\in S}$, $\{f_s\}_{s\in S}$,
satisfying Lie bracket relations:
\begin{enumerate}[(1)]
    \item $[\mfh,\mfh]=0$,
    \item $[h,e_s] = \alpha_s(h)\, e_s$ and $[h,f_s] = -\alpha_s(h)\, f_s$ for all $h \in \mfh(A)$ and $s\in S$,
    \item $[e_s,f_t] = \delta_{st}\, h_s$,
    \item $\ad(e_s)^{1-A_{st}}(e_t) = 0$ for all $s\neq t\in S$, and
    \item $\ad(f_s)^{1-A_{st}}(f_t) = 0$ for all $s\neq t\in S$.
\end{enumerate}
The algebra $\mfg(A)$ has a triangular decomposition $$\mfg(A) = \mfn^-(A) \oplus \mfh(A)
\oplus \mfn^+(A),$$ where $\mfn^{\pm}(A)$ are the subalgebras generated by $\{e_s\}_{s\in S}$ and $\{f_s\}_{s\in S}$ respectively. More specifically, the algebra $\mfn^{+}(A)$ is the
free Lie algebra generated by $\{e_s\}_{s\in S}$ satisfying the relations in (4) above.  Similarly, $\mfn^-(A)$ is freely generated by the $\{f_s\}_{s\in S}$ subject to the relations in (5).

In the proof of Proposition \ref{P:cartantoalg}, we want to be able to work with Kac-Moody Lie algebras without specifying a Cartan matrix.  To do this, consider the set the variables $\{a_{st}\}_{(s,t)\in S^2}$ and let $R_0 = \C[a_{st} : (s,t)\in S^2]$ denote the polynomial ring generated by these variables.  Let $\mcR$ be the free (associative non-commutative) $R_0$-algebra generated by symbols $\tilde{f}_s$, $\tilde{h}_s$, and
$\tilde{e}_s$, for $s\in S$. We use $\tilde{f}^I$ to denote
an arbitrary non-commutative monomial in the variables $\{\tilde{f}_s\}_{s\in S}$, $y^I$ to denote a
non-commutative monomial in the variables $\{\tilde{h}_s\}_{s\in S}$ and $\{\tilde{e}_s\}_{s\in S}$, and $x^I$
to denote a non-commutative monomial in all three families of variables.  A
general element of $\mcR$ can then be written as
\begin{equation*}
    \sum_I g_I(a)\, x^I,
\end{equation*}
where each $g_I(a)$ is an element of $R_0$, and $g_I(a)=0$ for all but finitely many $I$.  We say an element of $\mcR$ is \textbf{independent of $a_{st}$} if each
monomial in the coefficients $g_I(a)$ do not contain $a_{st}$ as a factor.  We say that an
element of $\mcR$ is a \textbf{normal form} if it can be written as
\begin{equation*}
    \sum_{I,J} g_{IJ}(a) \tilde{f}^I y^J.
\end{equation*}
In other words, every monomial is ordered so that all $\tilde{f}_s$'s precede
all $\tilde{h}_s$'s and $\tilde{e}_s$'s.  Given a Cartan matrix $A=[A_{st}]_{(s,t)\in S^2}$ over $S$, there is a morphism $\phi(A) : \mcR
\arr \mcU(\mfg(A))$ which sends
\begin{equation*}
    (\tilde{f}_s,\tilde{h}_s,\tilde{e}_s) \mapsto (f_s,h_s,e_s)\quad  \text{and}\quad
        a_{st} \mapsto A_{st} \text{ for all } (s,t)\in S^2.
\end{equation*}
Note that $a_{ss}$ is a variable in $R_0$. However, $\phi(A)(a_{ss})=2$ for all Cartan matrices $A.$  Hence the reader can replace $a_{ss}$ with 2 in all the proceeding calculations.

Let $\mfg'(A)$ denote the commutator subalgebra of $\mfg(A)$.  The morphism $\phi(A)$ maps $\mcR$ surjectively onto $\mcU(\mfg'(A))$, the universal
enveloping algebra of $\mfg'(A)$, since it is generated by $\{f_s, h_s, e_s \ |\ s\in S\}$. Given an element
$\tau \in \mcR$, we say that $\upsilon \in \mcR$ is \textbf{a normal form of}
$\tau$ if:
\begin{enumerate}
\item[(a)] $\upsilon$ is a normal form, and
\item[(b)] $\phi(A)(\upsilon) =\phi(A)(\tau)$ for all Cartan matrices $A$ over $S^2$.
\end{enumerate}

Using the relations (2) and (3) from the definition of $\mfg(A)$, it is clear that every element of $\mcR$ has a normal
form.\footnote{This is analogous to the isomorphism $\mcU(\mfg(A)) \iso
\mcU(\mfn^-(A)) \otimes \mcU(\mfh(A) \oplus \mfn^+(A))$. Working with $\mcR$
allows us to use this isomorphism without specifying the Cartan matrix $A$.}
Given  $\tau\in \mcR$, we construct a specific normal form $\eta(\tau)$
as follows:
\begin{itemize}
    \item if $\tilde{e}_s \tilde{f}_t$ or $\tilde{h}_s \tilde{f}_t$
        does not occur in any monomial of $\tau$, then set $\eta(\tau)=\tau$.
    \item Otherwise find the rightmost occurrence of $\tilde{e}_s \tilde{f}_t$
        or $\tilde{h}_s \tilde{f}_t$ in each monomial of $\tau$, and either
        \begin{itemize}
            \item replace $\tilde{e}_s \tilde{f}_t$ with $\tilde{f}_t \tilde{e}_s+ \delta_{st} \tilde{h}_s $, or
            \item replace $\tilde{h}_s \tilde{f}_t$ with $\tilde{f}_t \tilde{h}_s-a_{st} \tilde{f}_t$.
        \end{itemize}
    \item Repeat for each monomial in the resulting sum.
\end{itemize}
We call $\eta(\tau)$ \textbf{\emph{the} normal form} of $\tau$.

\begin{example}\label{Ex:normal_form}
Let $S=\{s_1,s_2,s_3\}$ and $\tau=\tilde h_1\tilde e_2\tilde e_3\tilde f_2$.  Applying the algorithm for $\eta(\tau)$ yields:
\begin{align*}
\tau=\tilde h_1\tilde e_2(\tilde e_3\tilde f_2) &\arr \tilde h_1\tilde e_2(\tilde f_2\tilde e_3)=\tilde h_1(\tilde e_2\tilde f_2)\tilde e_3\\
 & \arr \tilde h_1(\tilde f_2\tilde e_2+\tilde h_2)\tilde e_3= (\tilde h_1\tilde f_2)\tilde e_2\tilde e_3  +  \tilde h_1\tilde h_2\tilde e_3\\
 & \arr (\tilde f_2\tilde h_1-a_{12}\tilde f_2)\tilde e_2\tilde e_3  +  \tilde h_1\tilde h_2\tilde e_3\\
 & = \tilde f_2\tilde h_1\tilde e_2\tilde e_3-a_{12}\tilde f_2\tilde e_2\tilde e_3  +  \tilde h_1\tilde h_2\tilde e_3=\eta(\tau).
\end{align*}
\end{example}

\begin{lemma}\label{L:normal}
    Let $\tau$ be an element of $\mcR$ which is independent of $a_{st}$ with $s\neq t$.
    Suppose that, in every monomial of $\tau$, every $\tilde{f}_t$ occurs to the
    left of every $\tilde{h}_s$ and $\tilde{f}_s$. Then $\eta(\tau)$ is also
    independent of $a_{st}$.
\end{lemma}
\begin{proof}
    An $a_{st}$ is created in the above procedure when we switch $$\tilde{h}_s \tilde{f}_t\arr\tilde{f}_t \tilde{h}_s-a_{st} \tilde{f}_t$$
    and $\tilde{h}_s$ is created whenever we switch $$\tilde{e}_s \tilde{f}_s\arr\tilde{f}_s \tilde{e}_s+ \tilde{h}_s.$$
    When the hypothesis of the lemma holds, all $\tilde{h}_s$'s occur to the
    right of all $\tilde{f}_t$'s in every monomial of $\tau$. If we always
    apply the above steps to the rightmost occurrence of $\tilde{e}_r
    \tilde{f}_u$ or $\tilde{h}_r \tilde{f}_u$, then we never create a
    $\tilde{h}_s$ to the left of any $\tilde{f}_t$. Thus we will never
    create any $a_{st}$'s.  Since $\tau$ is independent of $a_{st}$,
    we conclude that $\eta(\tau)$ is also independent of $a_{st}$.
\end{proof}

Example \ref{Ex:normal_form} illustrates that if $\tilde{f}_t$ occurs to the right of $\tilde{h}_s$ in $\tau$, then $\eta(\tau)$ may dependent on $a_{st}$ even if $\tau$ does not.

For the rest of this section, we work with two Cartan matrices $A$ and $A'$ over finite sets $S$ and $S'$ with $|S|=|S'|$.  We will also fix a bijection $\sigma:S\arr S'$
and for the sake of notational simplicity, we let $A'_{st}:=A'_{\sigma(s)\sigma(t)}$ for all $s,t\in S$ (note that we are not assuming that $A_{st}=A'_{st}$). Let $\mfg(A)$ and $\mfg(A')$ denote the corresponding Kac-Moody algebras to $A$ and $A'$.  We will use the notation
above for the generators of $\mfg(A)$, and refer to the generators of
$\mfg(A')$ by $$h_s':=h_{\sigma(s)},\quad e_s':=e_{\sigma(s)},\quad\text{and}\quad f_s':=f_{\sigma(s)}.$$ Note that we always use $\mfg(A')$
instead of $\mfg'$ to avoid confusion with the commutator subgroup of $\mfg =\mfg(A)$.

Recall that a representation for $\mfg(A)$ is said to be integrable if $e_s$ and $f_s$ are locally nilpotent for every $s\in S$. Given a dominant integral weight
$\lambda$, we let $L^{\max}(\lambda) = L^{\max}(\lambda, A)$ denote the maximal
integrable $\mfg(A)$ module of highest weight $\lambda$.

Given two Cartan matrices $A$ and $A'$ over $S$ and $S'$ as above, we write $A \leq A'$ to mean that
$A_{st} \leq A'_{st}$ for all $s,t\in S$ (in other words, $|A_{st}| \geq |A'_{st}|$) .   For the next few lemmas, we assume the following hypotheses:
\begin{hypothesis}\label{H:quotient}
    \begin{enumerate}[(i)]
        \item $A$ and $A'$ are Cartan matrices over $S$ and $S'$ with $A \leq A'$.
        \item $\lambda \in \mfh(A)^*$ and $\lambda' \in \mfh(A')^*$ are dominant
            integrable weights such that $\lambda(h_s) \geq \lambda'(h_s')$ for
            all $s\in S$.
        \item $V = L^{\max}(\lambda, A)$ and $V' = L^{\max}(\lambda', A')$ are
            the maximal integrable modules with highest weights $\lambda$ and
            $\lambda'$.
        \item $\omega$ and $\omega'$ are highest weight vectors for $V$ and
            $V'$ respectively.
    \end{enumerate}
\end{hypothesis}
\begin{lemma}\label{L:quotient}
    Suppose hypotheses \ref{H:quotient} hold. Then:
    \begin{enumerate}[(a)]
        \item[(a)] There are surjective Lie algebra morphisms
        $$\psi^+ : \mfn^+(A) \arr \mfn^+(A')\qquad \text{and}\qquad \psi^- : \mfn^-(A) \arr \mfn^-(A')$$
        mapping $e_s \mapsto e_s'$ and $f_s \mapsto f_s'$ respectively.
        \item[(b)] There is a surjective $\mfn^-(A)$-module morphism
            $\pi : V \arr V'$ sending $\omega$ to $\omega'$, where $V'$ is
            regarded as a $\mfn^-(A)$-module using $\psi^-$.
    \end{enumerate}
\end{lemma}
\begin{proof}

    Part (a), the fact that $\psi^+ : \mfn^+(A) \arr \mfn^+(A')$ is well-defined and surjective follows from the fact that $1-A_{st} \geq 1-A'_{st}$ for all for all $s,t\in S$.  In particular, the relation $\ad(e_s')^{1-A_{st}}(e_t')=0$ holds in $\mfn^+(A')$.  Similar argument holds for the map $\psi^- : \mfn^-(A) \arr \mfn^-(A')$.

    For part (b), recall that $L^{\max}(\lambda, A)$ is the quotient of the
    Verma module (see \cite[Definition 2.1.1 and Definition 2.1.5]{Kumar02}) $$M(\lambda, A) := \mcU(\mfn^-(A)) \otimes_{\C} \C_{\lambda}$$ by the
    $\mfg(A)$-module $M^1(\lambda, A)$ generated by $f_s^{\lambda(h_s)+1}
    v_\lambda$, where $s\in S$ and $v_{\lambda}$ is the cyclic weight
    vector $1\otimes 1$ in $M(\lambda,A)$.  By \cite[Lemma 2.1.6]{Kumar02}, the action of $\mfn^+$ sends the generators $f_s^{\lambda(h_s)+1}
    v_{\lambda}$ to zero, so that $M^1(\lambda, A)$ is also the
    $\mcU(\mfn^-(A))$-module generated by $f_s^{\lambda(h_s)+1} v_{\lambda}$.  By
    part (a), for any cyclic weight vector $v_{\lambda'}$ of $M(\lambda', A')$,
    there is a surjective $\mcU(\mfn^-(A))$-module map $\pi : M(\lambda, A)
    \arr M(\lambda', A')$ sending $v_{\lambda} \mapsto v_{\lambda'}$. Now
    \begin{equation*}
        \pi(f_s^{\lambda(h_s)+1} v_{\lambda}) = (f'_s)^{\lambda(h_s)+1} v_{\lambda'}
    \end{equation*}
    belongs to $M^1(\lambda', A')$ for all $s\in S$ since $\lambda(h_s) + 1 \geq \lambda'(h_s')+1$.  Thus we get an
    induced morphism $\pi : L^{\max}(\lambda, A) \arr L^{\max}(\lambda', A')$
    on the quotients.
\end{proof}

We now recall the root system of a Kac-Moody algebra $\mfg(A)$ and the inversion set of an element $w\in W(A)$.  Let $Q = Q(A):=\bigoplus_{s \in S} \Z\, \alpha_s\subset \mfh^*$ denote the root lattice and let $R(A)\subseteq Q$ denote the root system of $\mfg(A)$.  We can decompose the set $R(A)=R^+(A)\sqcup R^-(A)$ where $R^+(A)$ and $R^-(A)$ denote the subsets of positive and negative roots respectively.    We have that $\mfn^{\pm}(A)=\bigoplus_{\alpha\in R^{\pm}(A)} \mfg(A)_{\alpha}$ where $\mfg(A)_{\alpha}$ is the root space corresponding to $\alpha$.  The Weyl group $W(A)$ (which is generated by $S$) acts on $Q$ by $$s(\alpha_t):=\alpha_t-\alpha_t(h_s)\, \alpha_s$$ for $s,t\in S$.  Given $w \in W(A)$, define the \textbf{inversion set}
$$I(w) := \{\alpha \in R^+(A) : w^{-1}(\alpha) \in R^-(A)\}.$$
Note that the inversion set is finite of size $\ell(w)$ and lies inside the sublattice $\bigoplus_{s \in S(w)} \Z\, \alpha_s$.
\begin{lemma}\label{L:rootsystem}
    Let $A$ and $A'$ be Cartan matrices over $S$ and $S'$ with $w\in W(A)$ and $w'\in W(A')$.  Suppose that $w$ is Cartan equivalent to $w'$ under the bijection $\sigma : S(w) \arr S(w')$. Then the induced isomorphism
    \begin{equation*}
        \sigma : \bigoplus_{s \in S(w)} \Z\, \alpha_s \arr \bigoplus_{s' \in S(w')}
            \Z\, \alpha_{s'}'
    \end{equation*}
    given by $\alpha_s \mapsto \alpha'_{\sigma(s)}$ identifies $I(w)$ with $I(w')$.
\end{lemma}
\begin{proof}
    Let $w = s_{1} \cdots s_{k}$ be a reduced expression. Then
    $I(w) = \{\beta_1,\ldots,\beta_k\}$, where
    \begin{equation*}
        \beta_\ell = s_{1} \cdots s_{\ell-1} (\alpha_{s_\ell}).
    \end{equation*}
    Given $1\leq \ell\leq k$, we want to show that $\sigma(\beta_\ell)$ belongs to $I(w')$.  First note that if $st \leq w$, then $A_{st} = A'_{\sigma(s)\sigma(t)}$ and hence
    $$\sigma(s(\alpha_t)) =\sigma(\alpha_t-A_{st} \alpha_s)=\alpha_{\sigma(t)}'-A'_{\sigma(s)\sigma(t)}\alpha_{\sigma(s)}'= \sigma(s)(\alpha'_{\sigma(t)}).$$
    In particular, $\sigma(s_{\ell-1}(\alpha_{s_\ell}))=\sigma(s_{\ell-1})(\alpha'_{\sigma(s_\ell)}).$
    For the purpose of induction, suppose that
    \begin{equation*}
        \sigma(s_{m} \cdots s_{\ell-1} (\alpha_{s_\ell})) = \sigma(s_m) \cdots \sigma(s_{\ell-1})(\alpha'_{\sigma(s_\ell)}).
    \end{equation*}
    Let $s_{m-1}=s$ and write $\displaystyle s_{m} \cdots s_{\ell-1} (\alpha_{s_\ell})=\sum_{t\in S} c_t\, \alpha_t$ where $c_t\in\Z$.  If $st\not\leq w$, then $t$ does not appear to the right of any $s$ in the reduced expression for $w$.  This implies the coefficient $c_t=0$ for all $t\in S$ such that $st\not\leq w$.  Hence
    \begin{align*}
        \sigma(s_{m-1}s_m \cdots s_{\ell-1} (\alpha_{s_\ell}))&=\sum_{t\in S} c_t\, \sigma(s_{m-1}(\alpha_t)) =\sum_{t\in S} c_t\, \sigma(s_{m-1})(\alpha'_{\sigma(t)})\\
        &= \sigma(s_{m-1})\left(\sum_{t\in S} c_t\,\alpha'_{\sigma(t)}\right) = \sigma(s_{m-1})\sigma(s_m \cdots s_{\ell-1}(\alpha_{s_\ell})).
    \end{align*}
    By induction, we have
    \begin{equation*}
        \sigma(\beta_\ell) = \sigma(s_{1} \cdots s_{\ell-1} (\alpha_{s_\ell}))=\sigma(s_1) \cdots \sigma(s_{\ell-1})(\alpha'_{\sigma(s_\ell)}).
    \end{equation*}
    Since $\sigma(s_1) \cdots \sigma(s_k)$ is a reduced expression for $w'$, we have $\sigma(\beta_\ell)\in I(w').$
\end{proof}
For $w \in W(A)$, define
\begin{equation*}
    \mfn^+(A)_w := \bigoplus_{\alpha \in I(w)} \mfg(A)_{\alpha}.
\end{equation*}
Since $I(w)$ is closed under bracket, $\mfn^+(A)_w$ is a finite-dimensional nilpotent Lie
algebra. We now add two additional hypotheses:
\begin{hypothesis}\label{H:cartaniso}
    \begin{enumerate}
        \item[(v)] Suppose that $w \in W(A)$ such that $A_{st} = A'_{st}$ for all $st \leq w$.
        \item[(vi)] Let $w' \in W(A')$ denote the element which is Cartan equivalent to $w$ under the bijection $\sigma:S\rightarrow S'.$
    \end{enumerate}
\end{hypothesis}
\begin{lemma}\label{L:nilpotent}
    Suppose that hypotheses \ref{H:quotient} and \ref{H:cartaniso} hold. Then
    $\psi^+$ induces an isomorphism $\mfn^+(A)_w \arr \mfn^+(A')_{w'}$.
\end{lemma}
\begin{proof}
    Because $\mfn^+(A)$ is generated by $\{e_s\}_{s\in S}$, it is spanned by Lie monomials in these same variables.  Since each monomial is contained in a root space, we have that $\psi^+( \mfg(A)_{\alpha})\subseteq  \mfg(A')_{\sigma(\alpha)}$ where $\sigma:S(w)\rightarrow S(w')$ is given in Lemma \ref{L:rootsystem}.  The lemma now follows from Lemma \ref{L:rootsystem}.
\end{proof}
\begin{lemma}\label{L:equivariant}
    Let $\exp:\mfg(A)\arr \mcG(A)$ denote the exponential map and suppose that hypotheses \ref{H:quotient} and \ref{H:cartaniso} hold.  Further suppose that $\lambda(h_s) = \lambda'(h'_s)$ for all $s\in S$.
    Then
    \begin{equation*}
        \pi\left( \exp(z)\, v \cdot \omega \right) = \exp(\psi^+(z))\, \sigma(v) \cdot \omega'
    \end{equation*}
    for all $z \in \mfn^+(A)$ and $v \leq w$.
\end{lemma}
\begin{proof}
    Let $v=s_{1}\ldots s_{k}$ be a reduced expression. If $st
    \not\leq w$ then $st \not\leq v$, so $\sigma(s_{1}) \ldots \sigma(s_{k})$ is
    also a reduced expression for $\sigma(v)$. By \cite[Definition 1.3.2 (5)]{Kumar02}, the action of
    $s$ (resp. $\sigma(s)$) on $V$ (resp. $V'$) is given by
    \begin{equation*}
        \exp(f_s) \exp(-e_s) \exp(f_s) \ \left(\text{resp.}\ \exp(f'_s) \exp(-e'_s) \exp(f'_s) \right).
    \end{equation*}
    Since $V$ is integrable, for any $s\in S$ and $u\in V$, there is an integer $N_0$ such that
    $$\exp(-e_s)\cdot u=\left(\sum_{k=0}^{N} \frac{(-e_s)^k}{k!}\right)\cdot u\quad\text{and}\quad\exp(f_s)\cdot u=\left(\sum_{k=0}^{N}\frac{f_s^k}{k!}\right)\cdot u$$
    for all $N\geq N_0$.  The same applies to $\exp(-e'_i)$ and $\exp(f'_i)$ when acting on $V'$. Since $V$ and $V'$ are highest-weight modules, we also have similar expressions for $\exp(z)$ and $\exp(\psi^+(z))$. Thus we can find an integer $N>>0$ such that
    \begin{align*}
        \exp(z)\, v \cdot \omega= \exp(z)(s_1\cdots s_k) \cdot \omega= \left(\sum_{j=0}^{N} \frac{z^{j}}{j!}\right) \cdot
            \left( \sum_{(l_1,m_1,n_1)\in[N]^3} \frac{f_{s_1}^{l_1}(-e_{s_1})^{m_1}f_{s_1}^{n_1}}{(l_1!)(m_1!)(n_1!)}\right) \cdots \\ \cdots  \left(\sum_{(l_k,m_k,n_k)\in[N]^3} \frac{f_{s_k}^{l_k}(-e_{s_k})^{m_k}f_{s_k}^{n_k}}{(l_k!)(m_k!)(n_k!)}\right)
            \cdot \omega
    \end{align*}
    and
    \begin{align*}
        \exp(\psi^+(z))\, \sigma(v) \cdot \omega'= \left(\sum_{j=0}^{N} \frac{\psi^+(z)^{j}}{j!}\right) \cdot
            \left( \sum_{(l_1,m_1,n_1)\in[N]^3} \frac{({f'}_{s_1})^{l_1}(-{e'}_{s_1})^{m_1}(f'_{s_1})^{n_1}}{(l_1!)(m_1!)(n_1!)}\right) \cdots \\ \cdots  \left(\sum_{(l_k,m_k,n_k)\in[N]^3} \frac{({f'}_{s_k})^{l_k}(-{e'}_{s_k})^{m_k}({f'}_{s_k})^{n_k}}{(l_k!)(m_k!)(n_k!)}\right)
            \cdot \omega',
    \end{align*}
    where $[N]:=\{0,1,\ldots N\}$.  Since $\mfn^+(A)$ is generated by $\{e_s\ |\ s\in S\}$, the element $z$ is a polynomial in $\{e_s\ |\ s\in S\}$. Since
    $\psi^+$ sends $e_s \mapsto e_s'$, we can thus write down an element $\Phi$
    in $\mcR$ such that
    \begin{equation*}
        \exp(z) v \cdot \omega = \phi(A)(\Phi)\cdot \omega \quad \text{ and } \quad
        \exp(\psi^+(z)) \sigma(v) \cdot \omega' = \phi(A')(\Phi)\cdot \omega'.
    \end{equation*}
    Now replace $\Phi$ with the normal form
    \begin{equation*}
        \eta(\Phi) = \sum_{I,J} g_{IJ}(a) \tilde{f}^I y^J.
    \end{equation*}
    Note that $h_s \omega = \lambda(h_s) \omega$ and $e_s \omega = 0$ for all $s\in S$. Hence, for any $I$ in the above sum, we have that $\phi(A)(y^I) \omega = c_I \omega$ for some $c_I \in \Z$. Since $\lambda(h_s) = \lambda'(h'_s)$, we also have $\phi(A')(y^I) \omega'
    = c_I \omega'$. Letting
    \begin{equation*}
        \eta_0 = \sum_{I,J} c_J g_{IJ}(a) \tilde{f}^I
    \end{equation*}
    yields
    \begin{equation*}
        \phi(A)(\Phi) \cdot\omega = \phi(A)(\eta_0)\cdot \omega \quad \text{and}
            \quad\phi(A')(\Phi) \cdot\omega' = \phi(A')(\eta_0) \cdot\omega'.
    \end{equation*}
    If $st \not\leq w$ then $st \not\leq v$, and this implies that
    $f_t$ occurs to the left of every $f_s$ in $\Phi$ by Lemma \ref{L:ij}.
    Since $h_s$ does not occur in $\Phi$, Lemma \ref{L:normal} implies that
    $\eta(\Phi)$ is independent of $a_{st}$. This also implies that $\eta_0$ is
    independent of $a_{st}$. On the other hand, if $st \leq w$, then
    $A_{st} = A'_{st}$. So
    \begin{equation*}
        \eta_0 = \sum_{I,J} c_J g_{IJ}(A) \tilde{f}^I = \sum_{I,J} c_J g_{IJ}(A') \tilde{f}^I.
    \end{equation*}
    It follows that $\phi(A')(\eta_0) = \psi^-(\phi(A)(\eta_0))$. Finally, the map $\pi$
    is $\psi^-$-equivariant, so
    \begin{align*}
        \pi\left( \exp(z) v \cdot \omega \right) & = \pi (\phi(A)(\Phi) \cdot \omega)
            = \pi(\phi(A)(\eta_0) \cdot \omega) \\
        & = \psi^-( \phi(A)(\eta_0)) \pi(\omega) = \phi(A')(\eta_0) \omega' \\
        & = \phi(A')(\Phi) \omega' =  \exp(\psi^+(z)) \sigma(v) \cdot \omega'.
    \end{align*}
\end{proof}
\begin{proof}[Proof of Proposition \ref{P:cartantoalg}]
    Let $A$ and $A'$ be Cartan matrices over $S$ and $S'$ respectively and that $\sigma:S(w)\arr S(w')$ gives a Cartan equivalence between $(w,A)$ and $(w', A')$. without loss of generality, we can assume that $S=S(w)$ and $S'=S(w')$ (for instance, see \cite[Lemma 4.8]{Richmond-Slofstra16}). In particular, hypotheses
    \ref{H:cartaniso} will hold. Assume (now with loss of generality) that
    $A \leq A'$. By \cite[Definition 7.1.19]{Kumar02}, the stable variety structure on
    $X(w,A)$ is induced (meaning that there is a closed embedding) by taking
    the map
    \begin{equation*}
        \mcG(A) / \mcB(A) \arr \mbP(V) : g \mapsto g \cdot \omega,
    \end{equation*}
    where $\mcB(A)$ is the Borel subgroup of $\mcG(A)$, and $\omega$ is the
    highest weight vector of $V = L^{\max}(\lambda, A)$ for a large enough
    dominant weight $\lambda$. By possibly increasing $\lambda$, we can assume
    that the variety structure on $X(w',A')$ is induced by taking the map
    \begin{equation*}
        \mcG(A') / \mcB(A') \arr \mbP(V') : g \mapsto g \cdot \omega',
    \end{equation*}
    where $\omega'$ is the highest weight of $V' = L^{\max}(\lambda',A')$
    for a dominant weight $\lambda'$ with $\lambda'(h_{\sigma(s)}) = \lambda(h_{s})$.
    Hence all parts of Hypothesis \ref{H:quotient} are satisfied.

    Consider the map $\pi : X(w, A) \arr \mbP(V')$ induced by $\pi : V \arr
    V'$. By \cite[6.2.E.1]{Kumar02}, every point of $X(w,A)$ can be written uniquely as
    $\exp(x)\, v \cdot\omega$ for some $v \leq w$ and $x \in \mfn^+(A)_v$. By Corollary
    \ref{C:cartan} and Lemmas \ref{L:nilpotent} and \ref{L:equivariant}, the map $\pi$
    induces a bijection $X(w,A) \arr X(w',A')$. Since $X(w',A')$ is a normal variety (see \cite[Theorem 8.2.2 (b)]{Kumar02}),
    the map $\pi$ restricted to $X(w,A)$ must be an algebraic isomorphism.

    Now suppose that $A\nleq A'$.  Define $\tilde{A}$ by $\tilde{A}_{st} = \min(A_{st}, A'_{st})$.
    By Lemma \ref{L:cartan}, there is an element $\tilde{w}$ in $W(\tilde{A})$
    such that $(\tilde w,\tilde A)$ is Cartan equivalent to both $(w,A)$ and $(w',A')$.
    We also have $\tilde{A} \leq A$ and $\tilde{A} \leq A'$ and hence the above argument implies $X(w,A) \iso X(\tilde{w},\tilde{A}) \iso X(w',A')$.
\end{proof}

\section{The cohomology ring of Schubert varieties}\label{S:cohomology}

In this section, we finish the proof of Theorem \ref{T:iso} by showing that
condition (2) implies condition (3), and that condition (3) implies condition
(1).  Let $A$ be a Cartan matrix over $S$, and let $w\in W(A)$. Recall that
$\mcB = \mcB(A)$ is the Borel subgroup of the Kac-Moody group $\mcG(A)$.  The
Schubert variety $X(w,A)$ has a stratification given by its decomposition into
Schubert cells
\begin{equation*}
    X(w,A)=\overline{\mcB w \mcB} / \mcB=\bigsqcup_{u\leq w} \mcB u\mcB/\mcB.
\end{equation*}
Let $x_v$ denote the fundamental class of the Schubert subvariety
$X(v,A)=\overline{\mcB v \mcB} / \mcB\subseteq X(w,A)$ in the integral homology
group $H_{2\ell(v)}(X(w,A)):=H_{2\ell(v)}(X(w,A),\Z)$.  Equivalently, if we
consider $X(v,A)$ as a cycle in the Chow group $\Ch_*(X(w,A))$, then
$$x_v=\cl(X(v,A))$$ where $\cl:\Ch_*(X(w,A))\arr H_*(X(w,A))$ denotes the cycle
map between Chow groups and homology.  It is well known that the Schubert
homology classes $\{x_v\}_{v\leq w}$ form a $\Z$-basis of $H_*(X(w,A))$.
The corresponding Schubert basis $\{\xi_v\}_{v\leq w}$ in integral cohomology
$H^*(X(w,A)):=H^*(X(w,A),\Z)$ is defined (using the identification
$H^*(X(w,A))\simeq \Hom_{\Z}(H_*(X(w,A)),\Z)$) by
\begin{equation*}
    \xi_v(x_u):=\delta_{vu}.
\end{equation*}
We now prove that condition (2) implies condition (3) in Theorem \ref{T:iso}.

\begin{prop}\label{P:algisomtocohom}
Let $A$ and $A'$ be Cartan matrices with $w\in W(A)$ and $w'\in W(A')$ and suppose that $\phi:X(w,A)\arr X(w',A')$ is an algebraic isomorphism.
Then the induced map
$$\phi^*:H^*(X(w',A'))\arr H^*(X(w,A))$$
is a graded ring isomorphism that identifies Schubert bases.
\end{prop}
\begin{proof}[Proof of Proposition \ref{P:algisomtocohom}]
Let $\Ch_*(X(w,A))$ and $\Ch_*(X(w',A'))$ denote the Chow groups of $X(w,A)$
and $X(w',A')$.  Since $\phi$ is an algebraic isomorphism, the induced
isomorphism of Chow groups $\phi_*:\Ch_*(X(w,A))\arr \Ch_*(X(w',A'))$ preserves
the cone of effective classes.  By \cite[Corollary to Theorem 1]{FMSS95}, any
effective class is a nonnegative $\Z$-linear combination of Schubert cycles.\footnote{This
idea effectively goes back to \cite{Hirschowitz84}; see also \cite{Kumar-Nori98}.}
Hence the Schubert cycles form the minimal extremal rays of the effective cone,
so  $\phi_*$ maps Schubert cycles to Schubert cycles. Since the Schubert classes
in cohomology are dual to the Schubert cycles in homology, $\phi^*$ maps
Schubert classes to Schubert classes.
\end{proof}
To finish the proof of Theorem \ref{T:iso}, we just need to show that condition (3) implies condition (1):
\begin{prop}\label{P:cohomtoCartan}
Let $A$ and $A'$ be Cartan matrices with $w\in W(A)$ and $w'\in W(A')$.  Suppose there is graded ring isomorphism $$\phi:H^*(X(w,A))\arr H^*(X(w',A'))$$ which identifies Schubert bases.  Then $(w,A)$ is Cartan equivalent to $(w',A')$.
\end{prop}
For the proof of Proposition \ref{P:cohomtoCartan}, suppose we have two
Schubert varieties $X(w,A)$ and $X(w',A')$ with an isomorphism $\phi$ between
cohomology rings as in the proposition.  Let $E:=\{\xi_v\}_{v\leq w}$ denote
the Schubert basis of $H^*(X(w,A))$.  As sets, there is a bijection $E\arr
[e,w]$ by mapping $\xi_u\mapsto u$.  Define $\tilde S(w):=E\cap H^2(X(w,A))$.
The bijection $E \arr [e,w]$ sends $\xi_s$ to $s$, and hence identifies $\tilde
S(w)$ with $S(w)$. We similarly define $E'$ and $\tilde{S}(w')$ for
$H^*(X(w',A'))$.  By assumption, the isomorphism $\phi$ restricted to $E$ gives
a degree-preserving bijection $\phi:E\arr E'$.  Let $\sigma : S(w) \arr S(w')$
be the bijection corresponding to $\phi|_{\tilde S(w)}$, so in particular
$\phi(\xi_s)=\xi_{\sigma(s)}$ for any $s \in S(w)$.  To prove Proposition
\ref{P:cohomtoCartan}, we will show that $\sigma$ is a Cartan equivalence. For
this, we need to show that all of the relevant entries of the Cartan matrix and
the reduced word structure of $w\in W(A)$ can be reconstructed from the ring
structure of $H^*(X(w,A))$.

Recall the set of roots $R(A)\subseteq Q(A)=\bigoplus_{s\in S}
\Z\alpha_s\subseteq \mfh^*(A)$.  For notational simplicity, we denote $R:=R(A)$
and $R^{\pm}:=R^{\pm}(A)$. A root $\beta\in R$ is said to be a real root if
$\beta=v(\alpha_s)$ for some $v\in W(A)$ and $s\in S$.  The Weyl group $W(A)$
acts on the dual space $\mfh(A)$ by
$$s(h):=h-\alpha_s(h)\, h_s$$
for $s\in S$ and $h\in\mfh(A)$, where $h_s \in \mfh(A)$, $s \in S$ are vectors
as in Section 3 (so $\alpha_t(h_s) = A_{st}$). The restriction of this action to $\bigoplus_{s \in S} \Z h_s$ is the dual action to the action of $W(A)$ on $Q(A)$.
 For a real root $\beta=v(\alpha_s)$, the corresponding coroot is
$\beta^{\vee}:=v(h_s)\in \mfh(A)$, and the corresponding reflection is
$s_{\beta}:=vsv^{-1}\in W(A).$ Note that $\alpha_t^{\vee}=h_t$ and
$s_{\alpha_t}=t$ for all $t\in S$.  Finally, we pick fundamental coweights
$\{\omega_s\}_{s\in S}\subseteq \mfh^*(A)$ satisfying the formula $$\omega_s(h_t):=\delta_{st}.$$
The main computational tool used to prove Proposition \ref{P:cohomtoCartan} is Chevalley's (non-equivariant) formula for multiplying Schubert classes by simple Schubert classes.
We use here the version from \cite[Theorem 11.1.7 (i)]{Kumar02}.
\begin{prop}\label{P:Chevalley}(Chevalley's formula)
For any $\xi_s\in \tilde S(w)$ and $\xi_u\in E$,
$$\xi_s\cdot \xi_u=\sum \omega_s(u^{-1}(\beta^{\vee}))\, \xi_{s_{\beta}u}$$ where the sum is over all real roots $\beta\in R^+$ such that $\ell(s_{\beta}u)=\ell(u)+1$ and $s_{\beta}u\leq w$.
\end{prop}
Note that if $\ell(s_{\beta}u)=\ell(u)+1$, then $u^{-1}(\beta^{\vee})\in R^+$ and hence $\omega_s(u^{-1}(\beta^{\vee}))$ is a nonnegative integer.  For any $F\in H^*(X(w,A)),$ define \textbf{the support of} $F$ as
$$\Supp(F):=\{\xi_v\in E : F(x_v)\neq 0\}.$$
In other words, if we write $F=\sum_{v\leq w} c_v\, \xi_v$, then
$\xi_v\in\Supp(F)$ if and only if $c_v\neq 0$. Let $\prec$ be the partial order
on $E$ generated by the covering relations $\xi_u\prec\xi_v$ for $u$ and $v$
such that $\xi_v\in \Supp(\xi_s\xi_u)$ for some $s\in S(w)$. In the next lemma, we show that $\prec$ corresponds
to Bruhat order $\leq$ on $[e,w]$:
\begin{lemma}\label{lem:bruhat}
$\xi_u\prec\xi_v$ if and only if $u\leq v$.
\end{lemma}
\begin{proof}
It suffices to consider covering relations.  Recall that $u\leq v$ if and only
if $v=s_{\beta}u$ with $\ell(v)=\ell(u)+1$ for some real root $\beta\in R^+$.
Hence $\xi_u\prec\xi_v$ immediately implies $u\leq v$.  Conversely, if $u\leq
v$ with $\ell(v) = \ell(u)+1$, then $v = s_{\beta}u$ where
$u^{-1}(\beta^{\vee})\in R^+$.  So there exists $s\in S(w)$ for which
$\omega_s(u^{-1}(\beta^{\vee}))>0$, implying that
$\xi_v\in\Supp(\xi_s\cdot\xi_u)$.  Hence $\xi_u\prec\xi_v$.
\end{proof}
Next we show that Cartan matrix entries $A_{st}$ for $st \leq w$ can
be recovered from products of simple Schubert classes.
\begin{lemma}\label{L:simple_products}
Let $s,t\in S(w)$ such that $s\neq t$.  Then:
\begin{enumerate}
\item  $\Supp(\xi_s\xi_t)\cap \Supp(\xi_t^2)\neq \emptyset$ if and only if $st\leq w$ and $st\neq ts$.  In this case,
    $$\Supp(\xi_s\xi_t)\cap \Supp(\xi_t^2)=\{\xi_{st}\}\quad \text{and}\quad A_{st}=-\xi_t^2(x_{st}).$$
\item $\Supp(\xi_s\xi_t)\cap (\Supp(\xi_t^2)\cup\Supp(\xi_s^2))=\emptyset$ if and only if $st=ts$.  In this case, $A_{st}=0.$
\end{enumerate}
\end{lemma}

\begin{proof}
By Chevalley's formula, we have
\begin{align*}
\xi_t^2 =\sum_{\substack{s' : s' t \leq w \\ s' \neq t}} \omega_t(t(\alpha_{s'}^{\vee}))\,\xi_{s't}+ \sum_{\substack{s'' : ts'' \leq w\\ ts'' \neq s'' t}} \omega_t(\alpha_{s''}^{\vee})\,\xi_{ts''}.
\end{align*}
In the second sum, since $t \neq s''$, $\omega_t(\alpha_{s''}^{\vee}) = 0$. So
\begin{align*}
\xi_t^2 &=\sum_{\substack{s' : s' t \leq w \\ s' \neq t}} \omega_t(\alpha_{s'}^{\vee}-A_{s't} \alpha_t^\vee)\,\xi_{s't}
=\sum_{\substack{s' : s't \leq w \\ s' \neq t}} -A_{s't}\,\xi_{s't}.
\end{align*}
Since $A_{s't} = 0$ if and only if $s't = ts'$, we conclude that
$$\Supp(\xi_t^2)=\{\xi_{s't}\ |\ s'\in S(w)\ \text{such that}\
s't\neq ts' \ \text{and}\ s't\leq w \}.$$

If $s \neq t$, then
\begin{align*}
    \xi_s \xi_t & =\sum_{\substack{s' : s' t \leq w \\ s' \neq t}} \omega_s(t(\alpha_{s'}^{\vee}))\,\xi_{s't}+ \sum_{\substack{s'' : ts'' \leq w\\ ts'' \neq s'' t}} \omega_s(\alpha_{s''}^{\vee})\,\xi_{ts''} \\
    & = \sum_{\substack{s' : s' t \leq w \\ s' \neq t}} \omega_s(\alpha_{s'}^{\vee}-A_{s't} \alpha_t^\vee)\,\xi_{s't} +\sum_{\substack{s'' : ts'' \leq w\\ ts'' \neq s'' t}} \omega_s(\alpha_{s''}^{\vee})\,\xi_{ts''}  \\
\end{align*}
If $st \leq w$ then the first sum is $\xi_{st}$, and otherwise the first sum is $0$.
If $st \neq ts$, and $ts \leq w$, then the second sum is $\xi_{ts}$, while otherwise
it's $0$. So we conclude that $\Supp(\xi_s\xi_t)\subseteq \{\xi_{st},
\xi_{ts}\}$ with $\xi_{st}\in \Supp(\xi_s\xi_t)$ if and only if $st\leq w$.
This proves part (1).  For part (2), note that either $st\leq w$ or $ts\leq w$.
So if $\Supp(\xi_t^2)\cup\Supp(\xi_s^2)$ contains neither $\xi_{st}$ or
$\xi_{ts}$, we must have $st=ts$.  Conversely, if $st=ts$, then
$\xi_{st}=\xi_{ts}$ cannot belong to either $\Supp(\xi_t^2)$ or
$\Supp(\xi_s^2)$. Since $\Supp(\xi_s\xi_t)$ will always contain one of $\xi_{st}$
or $\xi_{ts}$, this proves part (2).
\end{proof}
\begin{corollary}\label{C:simple_products}
Let $s,t\in S(w)$ such that $s\neq t$.  If $st\leq w$, then $A_{st}=A'_{\sigma(s)\sigma(t)}$.
\end{corollary}
\begin{proof}
First suppose that $s,t\in S(w)$ commute.  By Lemma \ref{L:simple_products} part (2), we have $$\Supp(\xi_s\xi_t)\cap (\Supp(\xi_t^2)\cup\Supp(\xi_s^2))=\emptyset.$$
Since $\phi$ is a graded ring isomorphism identifying Schubert bases, we also have that
$$\Supp(\phi(\xi_s)\phi(\xi_t))\cap (\Supp(\phi(\xi_t)^2)\cup\Supp(\phi(\xi_s)^2))=\emptyset.$$
This implies that $\sigma(s),\sigma(t)$ also commute and $A_{st}=A'_{\sigma(s)\sigma(t)}=0.$
If $s,t$ do not commute, then Lemma \ref{L:simple_products} part (1) implies
$$\Supp(\xi_s\xi_t)\cap \Supp(\xi_t^2)=\{\xi_{st}\}.$$
Applying the map $\phi$ to this equation gives
$$\Supp(\phi(\xi_s)\phi(\xi_t))\cap \Supp(\phi(\xi_t)^2)=\{\phi(\xi_{st})\}.$$
Since this set is nonempty, Lemma \ref{L:simple_products} part (1) also implies that $\sigma(s)\sigma(t)\leq w'$ and $\sigma(s), \sigma(t)$ do not commute. Moreover, we know
that
$$\Supp(\phi(\xi_s)\phi(\xi_t))\cap \Supp(\phi(\xi_t)^2)=
    \Supp(\xi_{\sigma(s)} \xi_{\sigma(t)}) \cap \Supp(\xi_{\sigma(t)}^2)
    = \{\xi_{\sigma(s)\sigma(t)}\},$$
so $\phi(\xi_{st}) = \xi_{\sigma(s)\sigma(t)}$. Since $A_{st}$ is the coefficient
of $\xi_{st}$ in $-\xi_{t}^2$, and $A_{\sigma(s)\sigma(t)}'$ is the coefficient
of $\xi_{\sigma(s)\sigma(t)}$ in $-\xi_{\sigma(t)}^2 = \phi(-\xi_{t}^2)$, we conclude that $A_{st}
= A_{\sigma(s)\sigma(t)}'$.
\end{proof}
One consequence of Lemma \ref{L:cartan} and Corollary \ref{C:simple_products}
is that if $w=s_1\cdots s_k$ is a reduced expression, then
$\sigma(s_1)\cdots\sigma(s_k)\in W(A')$ is a reduced expression Cartan
equivalent to $w$.  What remains to be shown is that
$w'=\sigma(s_1)\cdots\sigma(s_k)$.  In order to show this, we next show how to
reconstruct the descent sets of elements in $W(A)$ from the ring structure of
$H^*(X(w,A)$.  For any $J\subseteq S$, let $W_J$ denote Coxeter subgroup
generated by $J$ and let $W^J$ denote the set of minimal length coset
representatives for the cosets $W(A)/W_J$.  The right descent set of $u\in
W(A)$ is
$$D_R(u):=\{s\in S : \ell(us)=\ell(u)-1\}.$$
If $u\leq w$, then $D_R(u)\subseteq S(w)$.  It is well known that $s\in D_R(u)$
if and only if $u\notin W^{\{s\}}$, and that $D_R(u)$ is the set of simple reflections
$s$ for which there are reduced expressions $u = s_1 \cdots s_k$ with $s_k =
s$.  For any subset $J\subseteq S(w)$, let $H^J$ be the subring of
$H^*(X(w,A))$ generated by $\{\xi_s\ |\ s\in S(w)\setminus J\}$, and let
$$E^J:=\bigcup_{F \in H^J} \Supp(F)$$ be the support set of $H^J$ in $E$.

\begin{lemma}\label{L:coset_reps}
    The map $E\arr [e,w]$ given by $\xi_v\mapsto v$ restricts to a bijection
    $E^J\arr W^J\cap [e,w]$.
\end{lemma}
\begin{proof}
We first show that $\{v \in [e,w] : \xi_v \in E^J\}$ is a subset of $W^J$.
Suppose $u \in W^J \cap [e,w]$, $s \in S(w) \setminus J$.  By Chevalley's
formula, if $\xi_v \in \supp(\xi_s \xi_u)$ then $v=s_\beta u$ for
some real root $\beta\in R^+$ with $\ell(v)=\ell(u)+1$ and
$\omega_s(u^{-1}(\beta^{\vee}))\neq 0$. Suppose $v$ is of this form. If $v
\not\in W^J$, then there is some simple reflection $s' \in D_R(v) \cap J$,
and consequently there is a reduced expression $v = s_1 \cdots s_k$ for $v$
with $s_k = s' \in J$. But $u=s_1\cdots \hat{s_\ell}\cdots s_k$, where
$\hat{s_{\ell}}$ means that $s_{\ell}$ is omitted from the product, and $u \in
W^J$, so we must have that $\ell=k$, and ultimately $v = u s'$. Hence
$s_{u^{-1} \beta} = u^{-1} s_\beta u = u^{-1} v = s'$, and since the
correspondence between roots and reflections is a bijection, $u^{-1} \beta =
\alpha_{s'}$ and $u^{-1} \beta^{\vee} = \alpha_{s'}^{\vee}$.  Since $s \not\in
J$, $\omega_s(u^{-1} \beta^{\vee}) = \omega_s(\alpha_{s'}^{\vee}) = 0$, so if
$v \not\in W^J$, then $\xi_v \not\in \Supp(\xi_s \xi_u)$.

Because $H^*(X(w),A)$ is graded, $\xi_s \in E^J$ for $s \in S(w)$ if and only
if $s \in W^J$. We've shown that if $u \in W^J \cap [e,w]$, $s \in S(W)
\setminus J$, and $\xi_v \in \Supp(\xi_s \xi_u)$, then $v \in W^J$, so we
conclude by induction on degree that $\{v \in [e,w] : \xi_v \in E^J\} \subset W^J$.

For the converse, we show that $\{\xi_v : v \in W_J \cap [e,w]\}$ is a subset
of $E^J$ by induction on length. Indeed, suppose that $v\in W^J\cap[e,w]$, and
write $v=s_1\cdots s_k$. Let $u:=s_1v=s_2\cdots s_k$.
Clearly $u\in W^J$, and by induction we can assume that $\xi_u\in E^J$.
Since $\alpha_{s_1}$ is an inversion of $v$, we have that $\alpha_{s_1}\in
R^+\cap v R^-$.  Hence $v^{-1}(\alpha_{s_1})\in v^{-1}R^+\cap R^-$ which
implies $v^{-1}(-\alpha_{s_1})\in v^{-1}R^-\cap R^+$.  Now \cite[Exercise
1.3.E]{Kumar02} says that, since $v\in W^J$, the set $v^{-1}R^-\cap
R^+\subseteq R^+\setminus R^+_J$, where $R_J^+$ is the subset of the positive
roots in the span of $\{\alpha_{s'} : s' \in J\}$.  Thus
$$u^{-1}(\alpha_{s_1})=u^{-1}s_1(-\alpha_{s_1})=v^{-1}(-\alpha_{s_1})\in R^+ \setminus R_J^+.$$
We conclude that there exists $s\notin J$ such that
$\omega_s(u^{-1}(\alpha^{\vee}_{s_1}))\neq 0$. Consequently $\xi_v\in
\Supp(\xi_s\xi_u)$, implying $\xi_v\in E^J$.
\end{proof}
We now define
$$\tilde D_R(\xi_u):=\{\xi_s\in \tilde S(w) : \xi_u\notin E^{\{s\}}\}.$$
\begin{lemma}\label{L:descents}
    For any $u\leq w$, the bijection $\tilde{S}(w) \arr S(w)$ given by
    $\xi_s\mapsto s$ restricts to a bijection $\tilde D_R(\xi_u)\arr D_R(u)$.
\end{lemma}
\begin{proof}
    Since $s\in D_R(u)$ if and only if $u \not\in W^{\{s\}}$, the lemma follows
    immediately from Lemma \ref{L:coset_reps}.
\end{proof}

\begin{lemma}\label{L:reduced_word_setup}
For every $\xi_s\in \tilde D_R(\xi_v)$, there exists a unique $\xi_u\in E^{\{s\}}$ for which $\xi_v\in \Supp(\xi_s\xi_u)$.  Moreover, $us=v$.
\end{lemma}
\begin{proof}
Fix $\xi_s\in \tilde D_R(\xi_v)$.  By Lemma \ref{L:descents}, $s\in D_R(v)$ and hence $u=vs\leq v$.  The element $u$ is the unique element less than $v$ such that $\ell(v)=\ell(u)+1$ and $u\in W^{\{s\}}$.  Hence Lemma \ref{lem:bruhat} and \ref{L:coset_reps} imply $\xi_v\in \Supp(\xi_s\xi_u)$.
\end{proof}
Recall that if $w \in W(A)$, then $\Red(w)$ is the set of reduced words of $w$.
Lemma \ref{L:reduced_word_setup} implies we can inductively define $\widetilde\Red(\xi_v)$ by setting $\widetilde\Red(\xi_e):=\{\epsilon\}$ where $\epsilon$ denotes the empty sequence, and
$$\widetilde\Red(\xi_v):=\left\{(\xi_{s_1},\cdots, \xi_{s_m}, \xi_s)\quad :\quad \parbox{3in}{$\xi_s\in \tilde D_R(\xi_v), (\xi_{s_1},\cdots, \xi_{s_m})\in \widetilde\Red(\xi_{u})$ for $u$\\ the unique element such that $\xi_{u}\in E^{\{s\}}$ and\\ $\xi_v\in \Supp(\xi_s\xi_u)$}\right\}.$$

\begin{lemma}\label{L:reduced_words}
For any $v\in[e,w]$, the bijection $\tilde S(w)\arr S(w)$ given by $\xi_s\mapsto s$ induces a bijection between $\widetilde\Red(\xi_v)$ and $\Red(v)$.
\end{lemma}
\begin{proof}
$\Red(v)$ is the set of sequences of the form $(s_1,\ldots,s_k)$ where $s_k\in D_R(v)$ and $(s_1,\ldots, s_{k-1})\in\Red(vs_k)$.  The lemma follows by induction on length using Lemmas \ref{L:descents} and \ref{L:reduced_word_setup}.
\end{proof}

\begin{proof}[Proof of Proposition \ref{P:cohomtoCartan}]  Suppose that $\phi:H^*(X(w,A))\arr H^*(X(w',A')$ is a graded ring isomorphism which identifies Schubert classes.  In particular, the restricted map $\phi|_{H^2(X(w,A))}$ induces a bijection $\sigma:S(w)\arr S(w')$.  Let $w=s_1\cdots s_k$ be a reduced expression. Lemma \ref{L:reduced_words} implies that $(\xi_{s_1},\ldots, \xi_{s_k})\in\widetilde\Red(\xi_w)$.  Since $\phi$ is an isomorphism that identifies Schubert classes, $\phi(\Supp(F))=\Supp(\phi(F))$ for any $F\in H^*(X(w,A))$.  As a result, $\phi(E^J)=(E')^{\sigma(J)}$ for any $J\subset S(w)$ and hence $\phi(\tilde D_R(\xi_v))=\tilde D_R(\phi(\xi_{v}))$.  Also we have that $\xi_v\in \Supp(\xi_s\xi_u)$ if and only if $$\phi(\xi_v)\in \Supp(\phi(\xi_s)\phi(\xi_u))=\Supp(\xi_{\sigma(s)}\phi(\xi_u)).$$
So Lemma \ref{L:reduced_word_setup} implies $(\xi_{\sigma(s_1)},\ldots, \xi_{\sigma(s_k)})\in\widetilde\Red(\phi(\xi_w))$.  Since $\phi$ is graded, $\phi(\xi_w)=\xi_{w'}$, and hence $\sigma(s_1)\cdots\sigma(s_k)\in\Red(w')$ by Lemma \ref{L:reduced_words}.  Finally, Corollary \ref{C:simple_products} says that $A_{st}=A_{\sigma(s)\sigma(t)}$ for all $st\leq w$.  Thus $w$ and $w'$ are Cartan equivalent.
\end{proof}
Together, Propositions  \ref{P:cartantoalg}, \ref{P:algisomtocohom}, and \ref{P:cohomtoCartan} complete the proof of Theorem \ref{T:iso}.

\subsection{Constructing a presentation of a Schubert variety}\label{SS:cohomology2}

In section, we describe how to construct a presentation of a Schubert variety
from geometric data. Suppose we have a variety $X$ which is isomorphic to a
Schubert variety $X(w,A)$ for some $w\in W(A)$, but the element $w$ and matrix
$A$ are unknown. We want to find a Cartan matrix $A'$ and $w' \in W(A')$ such
that $X \iso X(w',A')$. This can be done if we know the integral cohomology ring
$H^*(X;\Z)$, and the effective cone $X$ in $H^*(X) = H^*(X;\Z)$. Indeed, let $E$ be the
generators of the extremal rays of the effective cone. Since $X$ is isomorphic
to a Schubert variety, we know that $E$ is the Schubert basis for $H^*(X)$.
The proof of Proposition \ref{P:cohomtoCartan} then gives a procedure to
construct $w'$ and $A'$ from $H^*(X)$ and $E$.  Specifically, for any $F \in
H^*(X)$, we can write $F=\sum_{\xi\in E} c_\xi\, \xi$ for some integers $c_\xi$ and
define
\begin{equation*}
    \Supp(F):=\{\xi\in E\ |\ c_\xi\neq 0\}.
\end{equation*}
Let $\tilde{S}:= E\cap H^2(X)$.  We first construct a Cartan matrix $A'$ over the set
$\tilde{S}$ using Lemma \ref{L:simple_products}. Let $\zeta_1,\zeta_2\in \tilde{S}$ and consider the
following cases (where products denote the product in $H^*(X)$):
\begin{enumerate}
    \item If $\zeta_1=\zeta_2$, then set $A'_{\zeta_1\zeta_2} =2$.
    \item If $\zeta_1 \neq \zeta_2$ and $\Supp(\zeta_1\zeta_2)\cap
        (\Supp(\zeta_1^2)\cup\Supp(\zeta_2^2))=\emptyset$, then set $A'_{\zeta_1\zeta_2}=A'_{\zeta_2\zeta_1}=0$.
    \item If $\zeta_1 \neq \zeta_2$ and $\Supp(\zeta_1\zeta_2)\cap\Supp(\zeta_2^2)=\{\nu\}$, then write
        $\displaystyle \zeta_2^2=\sum_{\xi\in E} c_\xi\, \xi$, and set $A'_{\zeta_1\zeta_2} = -c_\nu$.
    \item If $\zeta_1 \neq \zeta_2$, $\Supp(\zeta_1\zeta_2) \cap \Supp(\zeta_2^2) = \emptyset$, and
        $\Supp(\zeta_1\zeta_2) \cap \Supp(\zeta_1^2) \neq \emptyset$, then set
        $A'_{\zeta_1\zeta_2}$ to any negative integer.
\end{enumerate}
By Lemma \ref{L:simple_products}, exactly one of these cases must hold for each
pair $(\zeta_1,\zeta_2)\in \tilde{S}^2$. In addition, if $\xi : S \arr
\tilde{S} : s \mapsto \xi_s$, then $A_{st} = A'_{\xi_s\xi_t}$ for all $st \leq
w$.

Next, observe that for any $\xi \in E$, we can construct $\tilde{D}_R(\xi)$
and $\widetilde{\Red}(\xi)$ strictly in terms of $H^*(X)$ and $E$, without referring to
$w$ or $A$. Indeed, for any $J \subseteq \tilde{S}$ we can define $H^J$ to be
the subalgebra of $H^*(X)$ generated by $J$, and $E^J$ to be the set
$\bigcup_{F \in H^J} \Supp(F) \subseteq E$.  We can then define
$\tilde{D}_R(\xi)$ and $\widetilde{\Red}(\xi)$ as in the previous section,
except that instead of making a distinction between $\xi_s$ and $s$, we write
$\zeta \in \tilde{S}$ in place of both of them (so for instance, we'd define
$\tilde{D}_R(\xi) = \{\zeta \in \tilde{S} : \xi \not\in E^{\{\zeta\}}\}$).

Let $\xi$ be the unique element of $E$ of highest degree, let
$(\zeta_1,\ldots,\zeta_k) \in \widetilde{\Red}(\xi)$, and let $w' := \zeta_1
\cdots \zeta_k \in W(A')$. Then $\xi$ must be $\xi_w$, and if we write $\zeta_i
= \xi_{s_i}$, then $s_1 \cdots s_k$ is a reduced word for $w$ by Lemma
\ref{L:reduced_words}. Thus $\xi$ is a Cartan equivalence between $(w,A)$ and
$(w',A')$, and so $X = X(w,A) \iso X(w',A')$.

\bibliographystyle{amsalpha}
\bibliography{isom}

\end{document}